\documentclass[12pt,titlepage,amssymb]{amsart}
\usepackage{graphicx,amsmath,amsthm, amssymb}

\newcommand{\ben}{\begin{enumerate}} 
\newcommand{\een}{\end{enumerate}}

 \newcommand{\T}{ \mathrm{T}}

  \def \bc{\vskip 0cm \begin{center}}
\def \ec{\end{center} \vskip 0cm} 
\def\bea{\begin{eqnarray*}}
  \def \eea{\end{eqnarray*}} 
  \def\R{\mathbb{R}}
   \def\H{\mathbb{H}}
\def\C{\mathbb{C}} 
\def\Z{\mathbb{Z}} 
\def\O{\mathcal{O}}

\newtheorem{thm}{Theorem}[section]
\newtheorem{cor}[thm]{Corollary} 
\newtheorem{lem}[thm]{Lemma}

\theoremstyle{definition}
\newtheorem{Def}{Definition}[section]

\theoremstyle{remark}
\newtheorem*{rem}{Remark}

\begin{document} 
\pagestyle{plain}

\begin{center}
{\Large {\bf Templates for geodesic flows }}
\vskip 1cm
 {\large { Tali Pinsky }}
\vskip 2cm
\end{center}

\section*{Abstract}
The fact that the modular template coincides with the Lorenz template, discovered by Ghys, implies modular knots have very peculiar properties. We obtain a generalization of these results to other Hecke triangle groups. In this context, the geodesic flow can never be seen as a flow on a subset of $S^3$, and one is led to consider embeddings into lens spaces. We will geometrically construct homeomorphisms from the unit tangent bundles of the orbifolds into the lens spaces, elliminating the need for elliptic functions. Finally we will use these homeomorphisms to compute templates for the geodesic flows. This offers a tool for topologically investigating their otherwise well studied periodic orbits.

\section{Introduction}
\subsection{motivation}
The study of periodic orbits in dynamical systems is a basic problem with a long history. In a variety of
examples one would like to find periodic orbits, and to study their properties and general structure. This is of great importance, for instance, in taking the semi-classical limit of dynamical systems \cite{Quantization}, or for computing averages of observables in both classical and quantum chaotic systems
\cite{statistics}.

We are interested in the case of flows in three dimensional manifolds. A periodic orbit of such a flow is an embedding of $S^1$ into the 3-manifold, hence a knot, and one can ask
which knot types arise as periodic orbits for a certain flow, or what knot invariants (if any) do they share.

This problem was considered by Birman and Williams in \cite{Birwil1} and \cite{Birwil2},  who analyzed two examples. 
The first example is the flow associated with the famous Lorenz equations. Here it was shown that the family of knots arising as periodic orbits,  
the so called
\textquotedblleft Lorenz knots\textquotedblright has very special properties. As an example of their results we mention that Lorenz
knots are prime and every Lorenz link is a fibered link, and a positive braid. These results are based 
on the fact that all periodic orbits of the Lorenz flow are described by a simple combinatorial construction, called the template. The Lorenz template is given in Figure \ref{Lorenz_template} below, together with a typical periodic orbit. All periodic orbits of the equations arise as orbits in the Lorenz template. 

\begin{figure}[h]\centering \includegraphics[width=6cm]{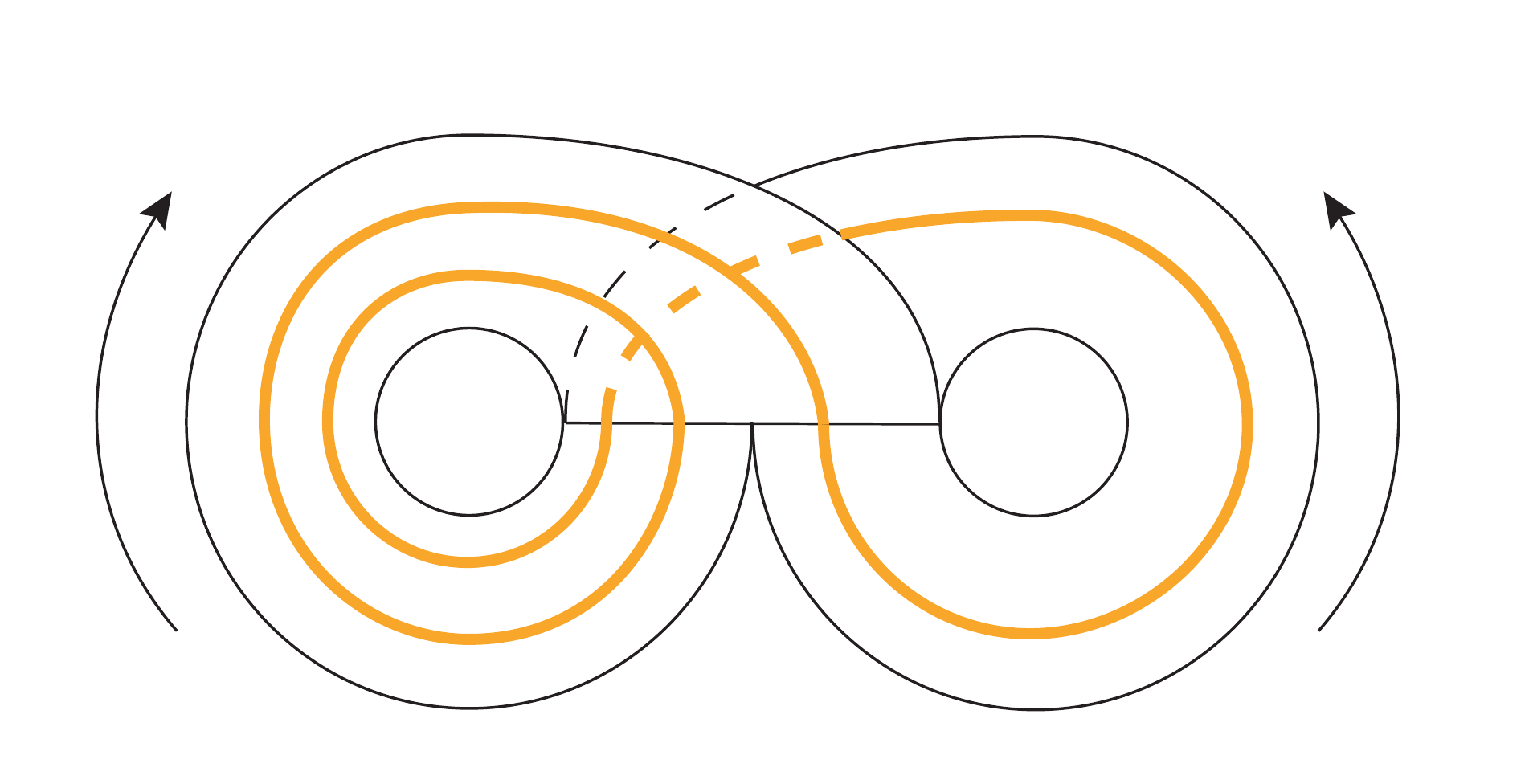} \caption{The Lorenz
template and a periodic orbit}\label{Lorenz_template} \end{figure}

The second example considered by Birman and Williams is the suspension flow on the complement of the figure eight knot in $S^3$. They have also constructed a template for the periodic orbits of this system. But, remarkably, here it was shown by Ghrist \cite{fig8} that the template is in fact a universal template : every possible knot in $S^3$ arises as a periodic orbit of this flow, without exception.  In the same paper Ghrist has shown that some templates termed Lorenz-like, studied before by Sullivan \cite{Sullivan} are also universal. In particular, even one half twist in one of the ears of the Lorenz template gives rise to a universal template, see Figure \ref{Lorenz_like_template}.

\begin{figure}[h]\centering \includegraphics[width=6cm]{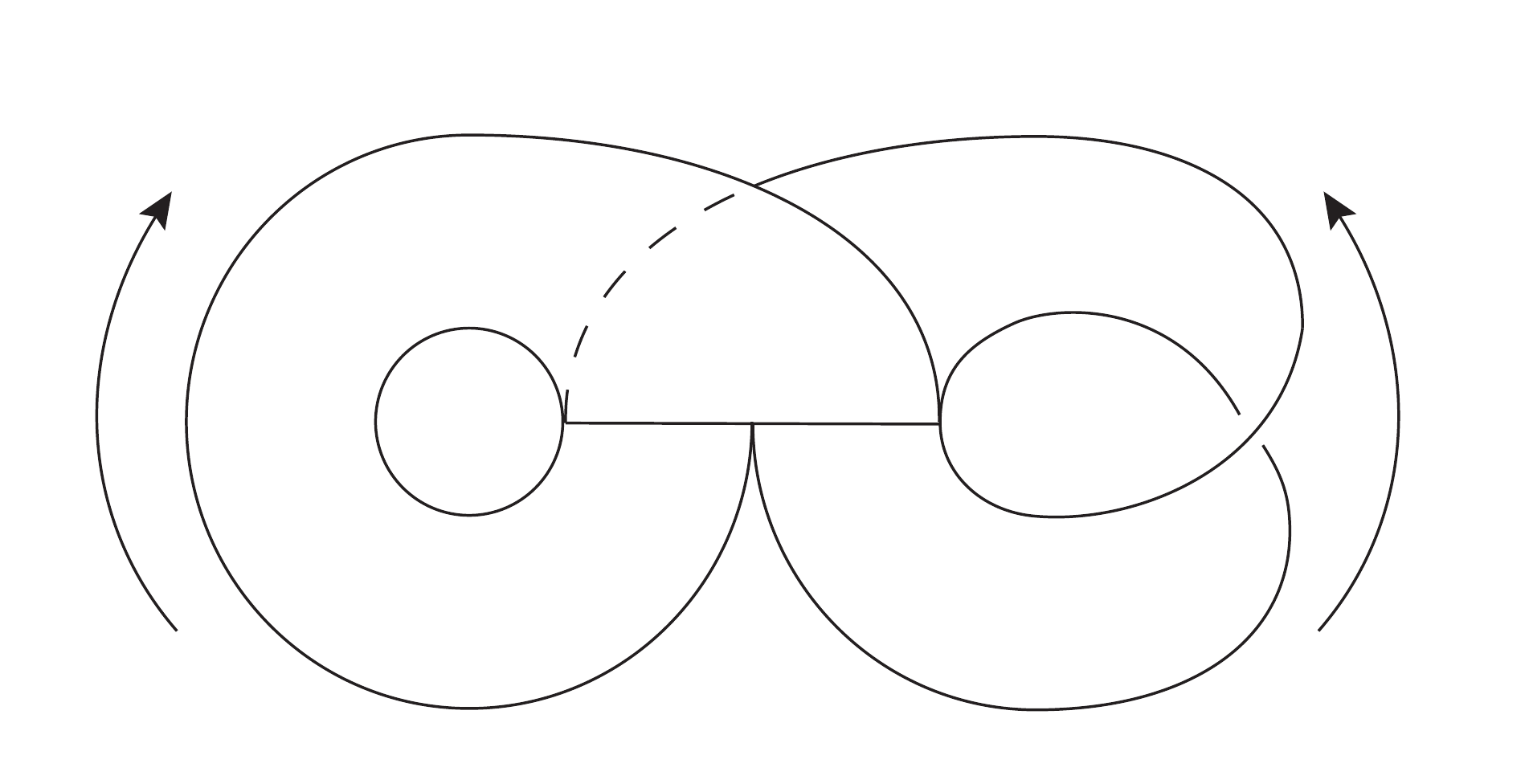}\caption{A universal Lorenz-like template}
\label{Lorenz_like_template} \end{figure}

Thus, constructing a template for a given flow  has far reaching consequences in the study of  knot types arising as periodic orbits, enabling a description of their knot invariants, or a proof of their universality. The construction, however, is necessarily very delicate, since even seemingly minute changes can dramatically affect the properties of the knots produced by the template. 

 In the important case of hyperbolic flows Birman and Williams have proved that a template always exists \cite{Birwil2}, but its explicit  construction in specific cases remains a difficult challenge. 
Even within the well-studied class of hyperbolic flows consisting of geodesic flows on the unit tangent bundle of surfaces of constant negative curvature,  the first construction of a template was achieved only recently, for the modular surface. Namely, Ghys \cite{Ghys} established the extraordinary fact that the modular template  coincides with  the
Lorenz template.  In particular, this fact implies that the modular knots share the special properties of the Lorenz knot family mentioned above. 

It is a compelling problem to understand which properties of the periodic orbits of the modular surface (if any) hold for the periodic orbits of the geodesic flows on other surfaces. 
The first step in the solution of this problem is constructing templates for these flows, and this will be our goal in the present paper. 

\subsection{Description of results}

We will focus on a class of surfaces which form a  natural generalization of the modular surface, namely the class of orbifolds with two cone points of orders $2$ and $k$ and one cusp, with $k$ odd. The modular surface   
appears as the first member of this family, with $k=3$. Alternatively, the surfaces in question arise 
as $\H^2/\Gamma(2,k)$ where $\Gamma(2,k)$ is the Hecke triangle group associated with the triangle $T(2,k,\infty)$. 

We will give a complete description of the template for the geodesic flow on the unit tanget bundle of these surfaces. Consideration of orbifolds with a cone point of arbitrary order gives rise to significant new phenomena and is considerably more complicated than the case of the modular surface. To begin with, let us note that Ghys' proof makes essential use of the fact that the unit tangent bundle of the modular surface can be identified with the complement of the trefoil knot in the unit sphere $S^3$. This identification is implemented via the embedding of the space of unimodular lattices 
in $\R^2$ 
into $P(\C^2)$ afforded by  the Weierstrass invariants $g_2$, $g_3$. However,  we will see that 
the modular surface is the only surface in our class whose unit tangent bundle can be identified with the complement of a knot in $S^3$. For all other surfaces in our class, the unit tangent bundle is identified with the complement of a knot in a non-trivial lens space. It is only the lift of this knot to $S^3$ which is the $(2,n)$ torus knot, generalizing the $(2,3)$  case of the modular surface. We note that the lens space in question is not determined uniquely, and in fact the unit tangent bundle of each $(2,n)$ orbifold embeds into a specific countable family of lens spaces, whose parameters are given as explicit functions of the Euler number of the bundle and $n$. 

The templates we construct are the ones arising for Euler number zero. For the case of the $(2,5)$ orbifold, for example, we obtain :

\begin{figure}[h]\centering \includegraphics[width=9cm]
{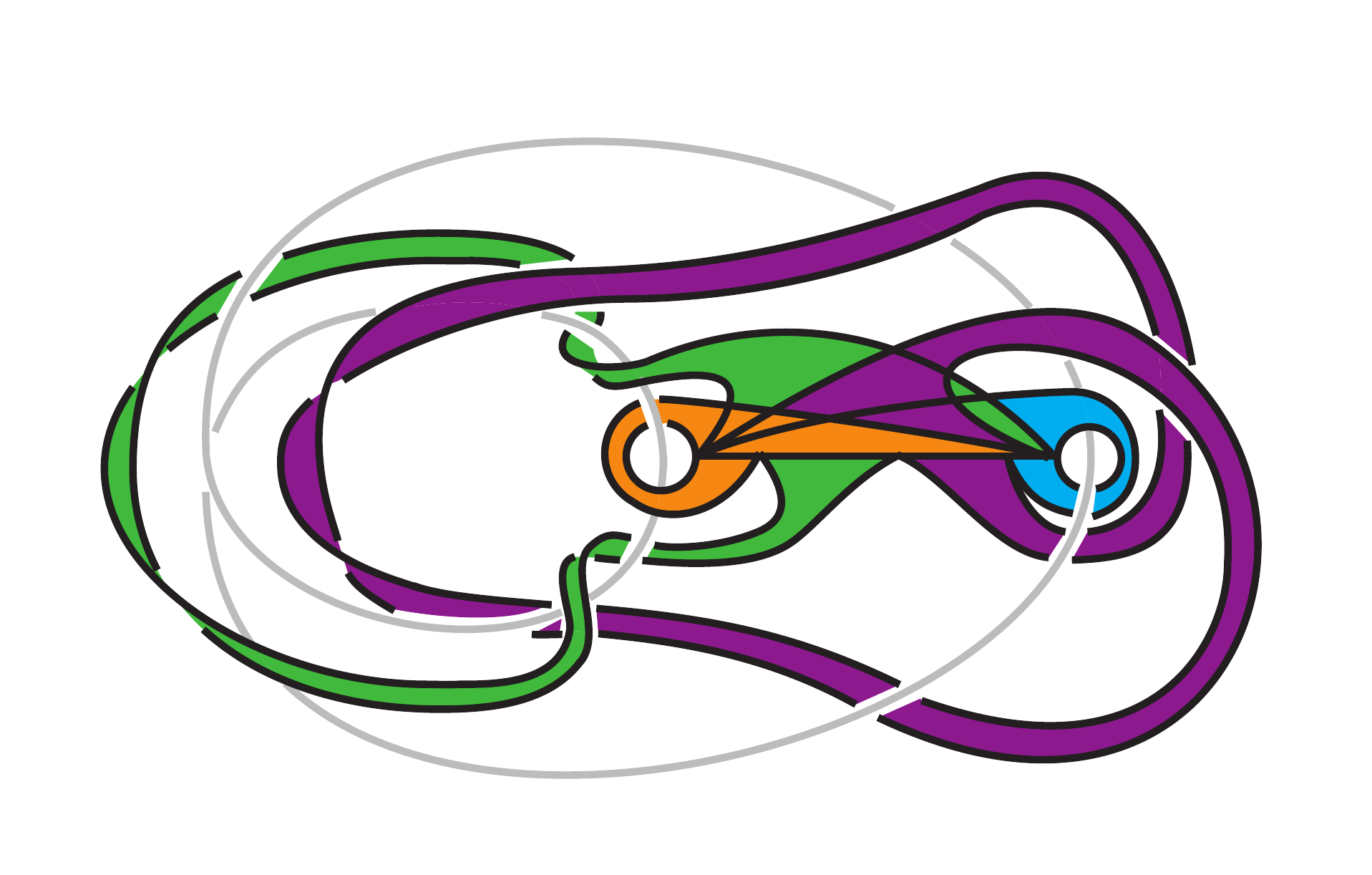}\end{figure}

Note that the template contains the Lorenz template as a subtemplate. This will be the case for the templates we construct for all the $(2,n)$ orbifolds, thus, the templates we construct are extensions of the basic Lorenz template. This is only due to the specific choice of Euler number zero, as other choices of non-zero Euler number will not give rise to a subtemplate with two unlinked ears each of which is the unknot. In a subsequent paper we plan to use the templates constructed here to prove all knots arising as periodic orbits are prime knots.

The case of $(2,n)$ with Euler number zero is also considered in the forthcomoing thesis of Pierre Dehornoy \cite{Pierre}, and any two orbits are found to be negatively linked.

Going back to the case of the modular surface, we note that here too a countable family of lens spaces is obtained in our construction, and it is a remarkable fact that the choice of Euler number zero is the choice producing the Lorenz template, and at the same time is  
the only one which gives the lens space $S^3$.
In particular, our approach gives a direct geometric proof of Ghys' result which makes no use of Weierstrass functions. 

\noindent {\bf Remark.}
The description of the geodesic flow on the modular surface as a flow in the complement of the trefoil knot in $S^3$, and the identification of the modular knots as Lorenz knots poses the following intriguing question. Is it possible to identify a Lorenz-invariant trefoil knot,  in the complement of which the Lorenz flow takes place?  We believe this is indeed the case, and specifically that the invariant curves connecting the two isolated fixed points of the Lorenz flow form a trefoil. We refer to \cite{whats_new} and \cite{Birman_review} for further discussion of this matter. 

\subsection*{Organization of the paper} The paper is organized as folllows: in \S 2 we introduce some preliminaries regarding hyperbolic geometry, in \S 3 after some preliminaries on lens spaces and Seifert fibered spaces we discuss the possible embeddings of the unit tangent bundle, in \S 4 we briefly describe the theory of templates and then construct the templates for the geodesic flows on the $(2,n)$ orbifolds. 

\subsection*{Acknowledgments} The author wishes to thank Professor Yoav Moriah and Professor Amos Nevo for numerous crucial discussions.
\section{Hyperbolic geometry}

\subsection{Hyperbolic unit tangent bundles\label{hutb}}
Let $\H$ be the hyperbolic plane. The tangent bundle $UT_p\H$ at any point $p\in\H$ is a plane. Now let $UT\H$ be the unit tangent bundle of $\H$, with fiber at each point consisting  of all vectors of norm 1, i.e., $S^1$. We call the fiber over a point the circle of directions at that point. As $\H$ is simply connected $UT\H$ is a trivial bundle, and can be described as the set of pointers
$(p,\theta)$ where $p\in\H$ and $0\leq\theta<2\pi$ is an angle representing the unit vector $e^{\theta i}\in S^1$ in the plane $T_p\H$.
Recall that $\text{Isom}_+(\H)\cong{PSL_2(\R)}$, and the stabilizer
$St(x)\subset PSL_2(\R)$ of a point $x\in\H$ consists of all rotations about $x$. $St(x)$ is thus naturally identified with the $S^1$ fiber above $x$ in $UT\H$, and in particular, ${PSL_2(\R)}\cong UT\H$.

Fix a pointer $(i,\pi/2)$ in $UT\H$. The homeomorphism is then explicitly given by 
\bea g:B\mapsto
(B^{-1}(i),dB^{-1}_i(\pi/2)) \eea for any $B\in PSL_2(\R)$. 

Any path $\gamma(t)$ in $\H$ has a natural lift to $UT\H$, by $\gamma(t)\mapsto\tilde\gamma(t)=(\gamma(t),\frac{\dot\gamma(t)}{||\dot\gamma(t)||})$. Let $\gamma(t)$ and $\delta(t)$ be two isotopic paths in $\H$. The isotopy may induce an isotopy between $\tilde\gamma(t)$ and $\tilde\delta(t)$, if the unit tangent vectors at any point $t$ are continuously deformed one to the other by the differential of the isotopy. In this case the isotopy is called a \emph{regular isotopy}, and can be regarded as an isotopy in $UT\H$. 

\subsection{The geodesic flow on $UT\H$\label{Geodesic_flows}}

The geodesic flow $\hat\phi_t$ on $\H$ is defined as the flow taking a pointer in $UT\H$ by parallel transport along the unique geodesic in $\H$ passing through it, a distance $t$.

Let us define for any $t\in\R$, 
$G_t= 
\left(\begin{array}{cc} 
e^{t/2} & 0 \\
0 & e^{-t/2}
\end{array} \right)$. 
$\hat\phi_t$ is conjugated by $g$ defined above to
a flow on $PSL_2(\R)$, given by $\phi_t(B)\mapsto G_t B$. 

Now, for any pointer in $UT\H$ we can consider a horocycle corresponding to the endpoint of the geodesic it defines, $h^-$, and the horocycle corresponding to the geodesic's starting point, $h^+$. We consider these horocycles in the unit tangent bundle, for $h^-$ with all directions pointing toward the endpoint, and for $h^+$ with all directions pointing away from the starting point, as in Figure \ref{horocycles}. 

\begin{figure}[h]\centering \includegraphics[width=6cm]{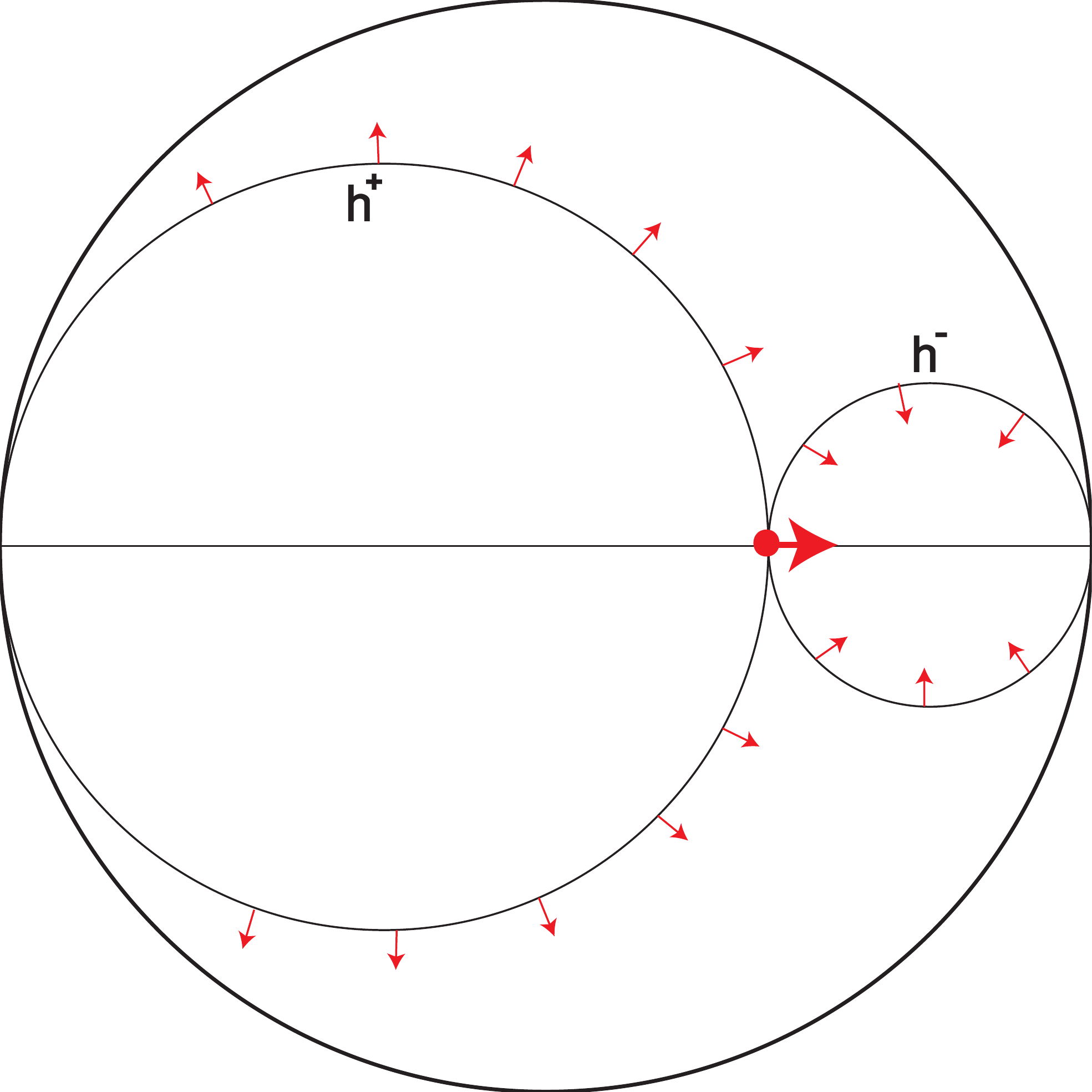} \caption{}\label{horocycles} \end{figure}

The one parameter group
$H^-_s= 
\left(\begin{array}{cc} 
1 & 0 \\
s & 1
\end{array} \right)$ defines the flow along the horocycle $h^-$, and the one parameter group
$H^+_u= 
\left(\begin{array}{cc} 
1 & u \\
0 & 1
\end{array} \right)$ along $h^+$. As $G_tH^-_sG_t^{-1}=H^-_{se^{-t}}$, we see that the images of points on the horocycle $h_-$ approach each other under the geodesic flow, exponentially fast. In the same manner, $G_tH^+_uG_t^{-1}=H^+_{ue^t}$ and so points on $h^+$ diverge exponentially fast under the geodesic flow. We also note that the geodesic and the two horocycles through a point in $UT\H$ are transversal, as can be seen in Figure \ref{horocycles}.


\subsection{The unit tangent bundle of an orbifold\label{orbifolds}}

Let $\O$ be an orbifold of dimension $n$. We will use only the case of dimension 2, but the definitions are the same for any dimension. Take an orbifold atlas for $\mathcal{O}$. Each chart in
the atlas is of the form $(U_i/G_i,\psi_i)$, where each $U_i$ is an open subset of $\R^n$, and $G_i$ is a
finite group acting linearly and faithfully on $\R^n$. 

The tangent bundle of $\O$ is the $2n$ dimensional orbifold defined by the charts
$((U_i,\R^n)/G_i,\tilde\psi_i)$ where $g\in G_i$ acts by $g((p,v))=(g(p),dg_p(v))$. 

If $\mathcal{O}$ is a good orbifold, i.e. a quotient of a manifold $M$ by a properly discontinuous group
action $\Gamma$, the tangent bundle $T\mathcal{O}$ is homeomorphic to $TM/\Gamma$. For the full
definitions see \cite{Montesinos}, page 92.

Let $\Gamma\subset PSL_2(\R)$ be a Fuchsian group acting on $\H$. The quotient $\Gamma\backslash\H$ is a two dimensional orbifold. It inherits a hyperbolic structure from $\H$, and its tangent bundle is $\Gamma\backslash T\H$. 
In this case, it is possible to define also the unit tangent bundle
$\Gamma\backslash UT\H\cong \Gamma\backslash PSL_2(\R)$, as the group acts by isometries. The geodesic flow on the unit tangent bundle is
defined to be $\tilde\phi_t(\overline{B})\mapsto \overline{B}H_t$, namely by projecting the flow $\phi_t$ defined above for $PSL_2(\R)$ via $\Gamma$.


\subsection{The $(n,k)$ Hecke triangle group \label{representation_variety}}

We now turn to the orbifolds which are the subject of our discussion. Consider the $(n,k)$ Hecke triangle group $\Gamma_{(n,k)}:=<v,u|v^k=u^n=e>$. By considering any two points of
distance $d>0$ in $\H$ and taking $v$ to be a rotation of order $2\pi/k$ about the first and $u$ a
rotation of order $2\pi/n$ about the other, we arrive at a representation of $\Gamma_{(n,k)}$ into
$PSL_2(\R)$. Any two representations corresponding to the same distance $d$ are conjugate since the isometry group is distance transitive.

For
any $n,k\in\Z$ there exists a distance $d_0$ for which the image of the representation is discrete, yet
the orbifold $\Gamma\backslash\H$ is of finite volume. The representations for which the distance equals $d_0$ are called lattice representations of $\Gamma$ and we denote the set of all such representations by $L(\Gamma)$.
Denote the orbifold
corresponding to any representation in $L(\Gamma)$ by $\O_{(n,k)}$. This orbifold has two cone points of order $n$ and $k$ and one
cusp. For a distance $d>d_0$ the volume of the orbifold
becomes infinite, that is  \textquotedblleft the cusp has opened\textquotedblright. Denote the orbifold corresponding to $d>d_0$ by $\O^d_{(n,k)}$.

A representation in $L(\Gamma)$ is
determined by one pointer: a cone point in $\H$ of order $k$ and any one of the $k$ directions equally spaced along
the circle of directions at that point, pointing to the neighboring $n-$cone points. We thus regard $L(\Gamma)$ as a set of pointers. Let us fix a representation $\Gamma_0=(i, \pi/2)\in L(\Gamma)$. That is, $\Gamma_0$ has a $k-$cone point at $i$ and an $n-$cone point at distance $d_0$ upwards along the imaginary axis. 
$\Gamma_0$ acts transitively on all $k-$cone points, and by rotations by ${2\pi}/{k}$ around $i$ (and any other $k-$cone point), hence $\Gamma_0$ normalizes itself and $L(\Gamma)\cong PSL_2(\R)/\Gamma_0$. It also true of course that $PSL_2(\R)/\Gamma_0\cong\{\text{pointers in the plane}\}/\Gamma\cong UT\O_{(n,k)}$ as in \S  \ref{orbifolds}. 

We will use in the sequel the following homeomorphism: 
\bea
h: PSL_2(\R)/\Gamma_0\rightarrow L(\Gamma)\\
h:B\cdot \Gamma_0\mapsto(B(i), dB_i(\pi/2)). \eea
Note that although $h$ is very similar to $g$ defined in Section \ref{hutb}, the action of $h$ is on the quotient spaces and in the \textquotedblleft opposite direction\textquotedblright.

\section{The unit tangent bundle}\ 

 As noted in the introduction, the unit tangent bundle to the modular orbifold
$\O_{(2,3)}$ is homeomorphic to the complement of a trefoil in $S^3$. The homeomorphism is implemented by the Weierstrass invariants, and was used by Ghys in \cite{Ghys} to compute the template of the modular flow. However, the fact that unit tangent bundle is a subset of $S^3$ is in fact unique to the modular surface among the surfaces we consider, as we will see below. 

We will give a direct geometric description of the unit tangent bundle to $\O_{(n,k)}$
as a three manifold. We will subsequently describe the template for the geodesic flow embedded therein. We note however, that the best one can attain for the general case is that  for any orbifold $\O_{(n,k)}$, the unit tangent bundle has (infinitely many) embeddings into lens spaces. These are parametrized by the Euler number, and we will see later on that 
the choice of the embedding with Euler number zero gives rise to a template generalizing the Lorenz template. 

This chapter is organized a follows : we begin in \S \ref{Lens} and \S \ref{Seifert} with  brief reminders on lens spaces and Seifert fibered spaces, and in \S 3.3 we describe the structure of $UT\O_{(n,k)}$ as a Seifert fibered space. In \S 3.4 we define a certain vector field on the orbifolds  $UT\O_{(2,k)}$. In \S 3.5 we parametrize the embeddings by the Euler number, and finally in \S 3.6 we
describe
 explicitly the embeddings of $UT\O_{(2,k)}$ into the relevant lens spaces, given by the vector field.
\subsection{Lens spaces \label{Lens}}
The Lens space $L(p,q)$, for integers $0<q<p$, is defined to be the quotient of $S^3$ by the following action of $\Z_p$. Let $S^3$ be the vectors of norm 1 in $\C^2$, then any $w\in\Z_p$ acts by
$w\cdot(z_1,z_2)=(w\cdot z_1,w^q\cdot z_2)$. This action is free, thus the resulting space $L(p.q)$ is a compact three dimensional manifold.

Another description of the same manifold is as follows: Let $\T_1$ and $\T_2$ be two solid tori. Fix a longitude and meridian generators $l_i$ and $m_i$ for $\pi_1(\partial\T_i)$. Then choose an orientation reversing  homeomorphism $h:\partial\T_1\rightarrow\partial\T_2$ so that $h_*(m_1)=pl_2+qm_2$. Then the lens space $L(p,q)$ is given  by gluing $\T_1\cup_h\T_2$.

The homeomorphism between the two representations can be seen by a straightforward identification as follows. Consider the polar parameterization $(r_1,\theta_1,r_2,\theta_2)$ of $\C^2$. For a point in $S^3$, $r_1^2+r_2^2=1$. This gives a parameterization $(r,\theta_1,\theta_2)$ for $S^3 $, $0\leq r\leq1$, $0\leq\theta_1,\theta_2<2\pi$, with obvious identifications for $r=1$ and $r=0$. This is the well known parameterization of $S^3$ by concentric tori sketched in Figure \ref{tori}. The radii are invariant under the $\Z_p$ action and so $r$ is invariant under the action, thus each of the tori is an invariant set. The action on each torus is given by $e^\frac{2\pi mi}{p}\cdot(r,\theta_1,\theta_2)\mapsto(r,\theta_1+\frac{2\pi m}{p},\theta_2+\frac{2\pi mq}{p})$. 
Fix some $0<r_0<1$. This gives a decomposition of $S^3$ to two solid tori. For the inner solid torus, i.e. the union of tori $0\leq r \leq r_0$ a meridian is given by $\theta_1=0$. As can be seen from the action,  the quotient is the torus resulting from $0<\theta_1\leq\frac{2\pi}{p}$, by identifying the two disks in the boundary with a twist by $\frac{2\pi q}{p}$. Thus, $\theta_1=0$ remains a meridian for the quotient torus. This meridian is a longitude for the other torus in $S^3$, $r_0<r\leq 1$. The image of the longitude is a $(p,q)$ curve in the quotient of the outer torus, and the equivalence of the definitions follows. See also \cite{Rolfsen}, \cite{lens}.

\begin{figure}[h]\centering \includegraphics[width=5cm]{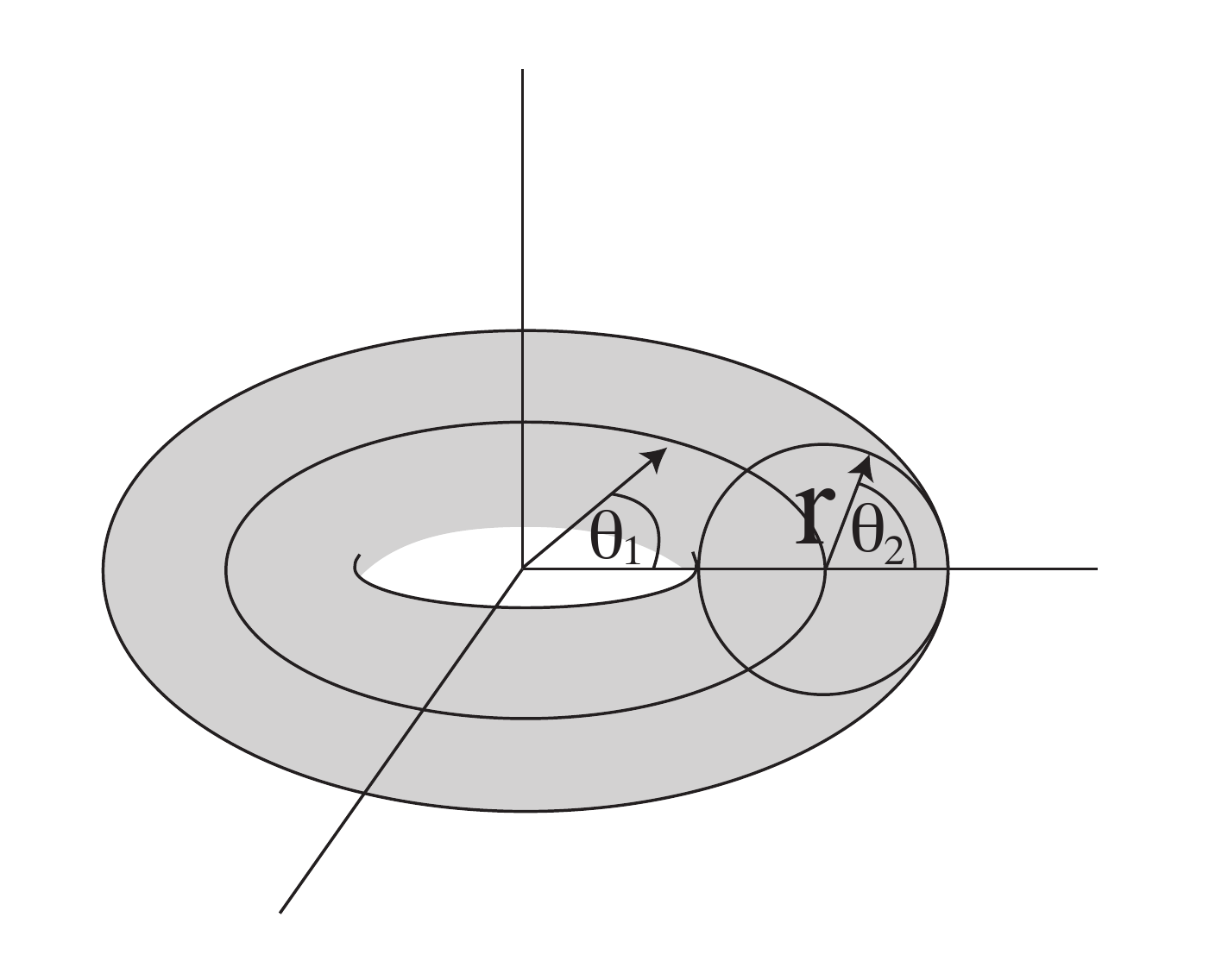} \caption{}\label{tori}
\end{figure}

\subsection{Seifert fiber spaces\label{Seifert}}
 In the section below we recall some terminology and facts about Seifert fiber spaces which can also be found in (\cite{Montesinos}, \cite{Seifert}).
 \begin{Def}
An orientable $3$-manifold is called a \emph{Seifert fiber space} if it is a disjoint union of fibers homeomorphic to $S^1$, such that each fiber has a solid torus neighborhood, foliated by fibers which are not meridians for it.
\end{Def}

\begin{Def}
A \emph{Seifert torus} of type $(\mu,\nu)$, $\mu$ and $\nu$ coprime integers,  is the torus obtained from a fibered cylinder $D^2\times [0,1]$ where the fibers are the lines $x\times [0,1]$, by identifying $(x,1)$ with $(r_{\nu/\mu}(x),0)$ for every $x\in D^2$. $r_{\nu/\mu}:D^2\rightarrow D^2$ is given by a rotation by angle $2\pi\nu/\mu$.
Without loss of generality, we can assume that $\mu>0$ and $0\leq\nu\leq\frac{1}{2}\mu$.
\end{Def}

In addition to a standard choice of a meridian longitude basis for the homology of the boundary of a Seifert torus, one can choose a basis of a fiber, and any simple closed curve on the boundary which intersects any fiber only once. Such a curve is called a \emph{Crossing curve}.

Any fiber $f$ in a general Seifert fiber space, has a neighborhood homeomorphic to some Seifert torus, the homeomorphism taking $f$ to the central fiber. The type of the Seifert torus is uniquely determined (taking the invariants normalized as in the theorem) and the fiber is called singular if $\nu\neq0$. In a $(\mu, \nu)$ fibered torus, any regular fiber $f$ is homeomorphic to $(f_0)^{\mu}$
where $f_0$ is the singular fiber. This will be used later to identify the invariants of specific Seifert tori.

By identifying every fiber to a point one gets a map from any Seifert fiber space $M$ to a two dimensional orbifold $S$ called the \emph{orbit surface}.
$S$ cannot in general be embedded into $M$. Each cone point of $S$ corresponds to a singular fiber of the Seifert fiber space. One can embed a subset of $S$ into $M$ in the following way. First, if $M$ and $S$ are closed, we remove a toral neighborhood of any regular fiber in $M$, and the corresponding disk in $S$, creating one boundary component $J_0$ for $S$. Then we remove a small neighborhood of each cone point of $S$, obtaining the \emph{punctured orbit surface} of $M$ we denote by $S_0$. Removing the corresponding toral neighborhoods of the singular fibers in $M$, we get a three manifold $M_0$ which is a bundle over $S_0$. Seifert, \cite{Seifert}, proves $S_0$ can always be embedded into $M_0$, and the embedding is unique once the homology types of the boundary curves of $S_0$ on $\partial M_0$ are determined.

\begin{thm}[Seifert] Any closed Seifert fiber space is uniquely determined by invariants
\bea
\{O/N,o/n,\, g,\, b\, ;\,  (\mu_1,\nu_1),\hdots,(\mu_s,\nu_s)\},
\eea
where one puts an $O$ if $M$ is orientable and $N$ if not, and $o$ if $S$ is orientable, $n$ if not. $g$ is the genus of $S$, $b$ is the Euler number of $M$. $s$ is the number of the singular fibers, and a toral neighborhood of the singular fiber $f_i$, $1\leq i\leq s$ is a Seifert torus of type $(\mu_i,\nu_i)$.
\end{thm}

Given the invariants, the space $M$ can be constructed as follows. Begin with a surface of genus $g$ with $s+1$ punctures. This is homeomorphic to $S_0$. Take the trivial circle bundle over $S_0$, This is a fibered space with $s+1$ toral boundary components. The boundary curves of $S_0$ on each of the boundary tori of $M_0$ are crossing curves for these tori. So for each boundary torus this determines a basis of a crossing curve $c_i$, $1\leq i\leq s$, and a fiber $f$. Thus the invariants $\mu_i,\nu_i$ for each singular torus uniquely determine a gluing of this singular torus to $M_0$ such that fibers match. Now there remains a single boundary component, with a given crossing curve $c_0$. We glue in a $(1,0)$ fibered torus by gluing its meridian to a $c_0-b\cdot f$ curve. 

In the sequel we will consider open Seifert fiber spaces, having a single toral boundary. As by gluing in a Seifert torus one obtains a closed Seifert fiber space, it follows from Seifert's theorem that these manifolds are uniquely determined by the above invariants, excluding the Euler number $b$. $b$ remains undetermined as the parameter determining the gluing, and parametrizes the $\aleph_0$ closed Seifert fiber space into which the manifold can embed without adding singular fibers.
 
On each Seifert fiber space, one can define an $S^1$ action, moving each point along the fiber containing it. Consider the action on a neighborhood of a singular fiber, which is a $(\mu,\nu)$ fibered torus. Take a meridional disk $D$ for this torus, thus the torus is given by $D\times[0,1]$ with the identification by a $2\pi\nu/\mu$ rotation. 
Choose a point $x_0$ on the meridian $\partial D$. There are $\mu$ intersection points of the orbit of $x_0$ (i.e. the fiber containing $x_0$) with $\partial D$. We order these points as $x_0,\hdots, x_{\mu-1}$, by their order along the meridian.
The flow takes $x_0=x_0\times 0$ along $[0,1]$ to a point $x_0\times1$ identified under the rotation with $x_{\nu}$. Thus, by knowing the action one can derive both $\mu$ (by the number of points in the orbit), and $\nu$. This will be useful later on.

\subsection{The structure of the unit tangent bundle} 

\begin{thm} The unit tangent bundle to $\O_{(n,k)}$ is a Seifert fiber space consisting of two Seifert tori of invariants $(k,1)$
and $(n,1)$. \end{thm}

\begin{proof} We divide the proof into several steps. 

Step 1 : First we prove the unit tangent bundle is a union of two Seifert tori. For this, divide the orbifold itself into two disks, each a neighborhood of one of the cone points, so that they intersect
along a line as in Figure \ref{division}.

\begin{figure}[h]\centering \includegraphics[width=5cm]{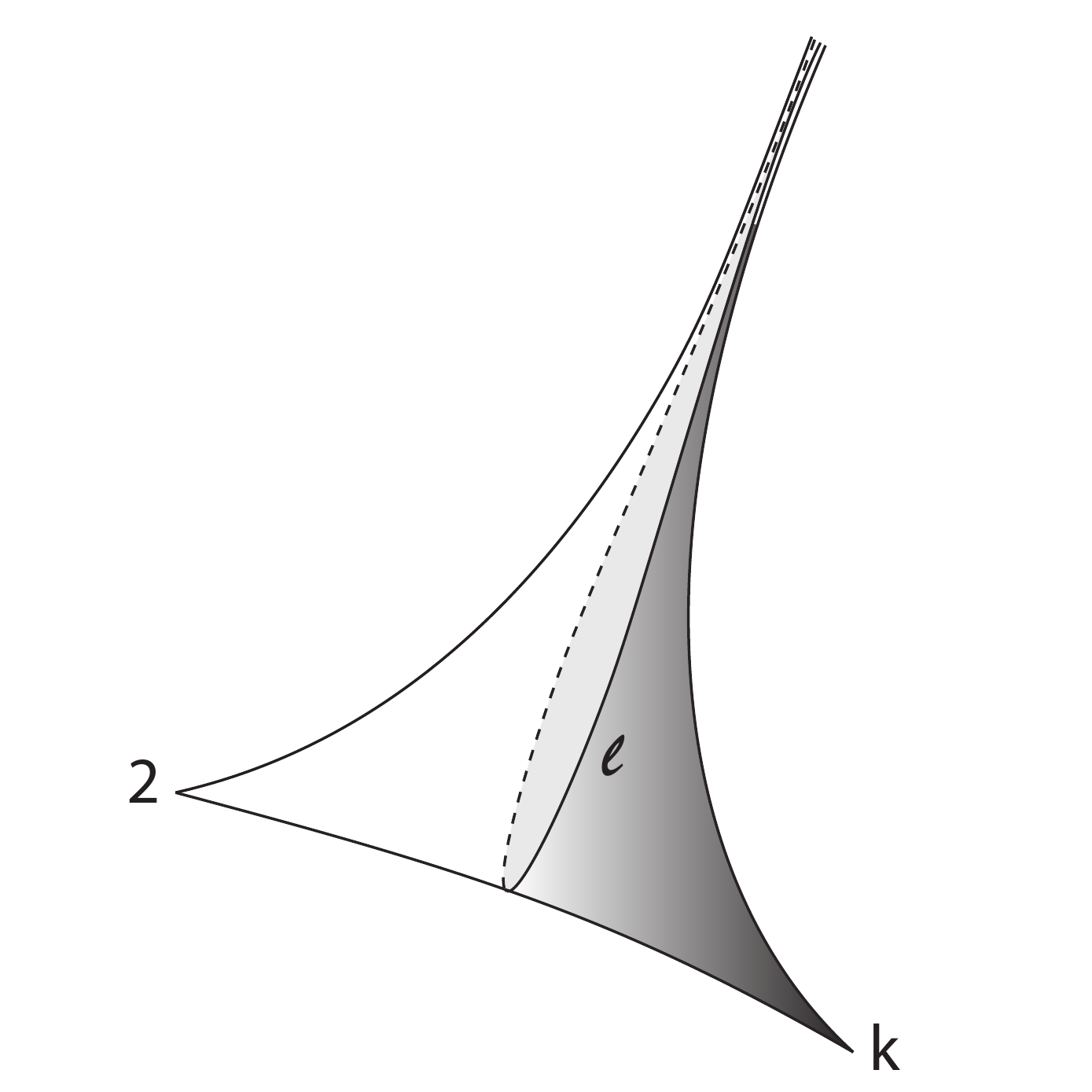} \caption{}\label{division}
\end{figure}

The unit tangent bundle to each of the disks is an  $S^1$ bundle over the disk, i.e., a fibered solid
torus. It remains to prove that only one fiber in each torus is singular. 

For a two dimensional orbifold, a path making a small loop around a point together with the tangent
direction at every point, is isotopic to the fiber corresponding to that point (we 
orient both counter clockwise).  Any two small loops on the surface are regularly isotopic, unless one of
them encloses a cone point, hence all fibers except at the cone points are isotopic and the torus contains at most one singular fiber, the one corresponding to the cone point. We denote the singular fiber corresponding to the $n$-cone point by $\alpha$ and the one corresponding to the $k$-cone point by $\beta$.

As in \ref{Seifert}, such a torus is determined up to homeomorphism  by two natural numbers $(\mu, \nu)$, $0\leq\nu\leq\mu/2$. Denote the invariants of the torus which is the unit tangent bundle to a neighborhood of the $n$-cone point $(\mu_n,\nu_n)$, and the invariants of the other torus by $(\mu_k,\nu_k)$. We now compute these invariants in two further steps.

Step 2 : computing $\mu_n$ and $\mu_k$. Consider the universal covering space of the orbifold. This is a $2-$plane and in it any two loops are isotopic. Hence, a lift of the singular fiber $f_0$ corresponding to the $n$-cone point (which projects to the singular fiber to the order $n$) and a lift of any of the regular fibers $f$ are isotopic. Thus in the projection, $f\cong f_0^{n}$. In the same way, the $k-$singular fiber to the $k^{th}$ order is isotopic to a regular fiber.
Hence as in \ref{Seifert}, the two unit tangent bundles to the cone points
neighborhoods are two tori  $(k, \nu_k)$ and  $(n, \nu_n)$.

Step 3 : computing $\nu_n$ and $\nu_k$. For the case $n=2$ which will be our main interest it is obvious that $\nu_n$ must equal 1 by the
above normalization. It is actually true in general that both $\nu_k$ and $\nu_n$ equal 1. We verify this following Montesinos \cite{Montesinos} by considering the $S^1$ action on the unit tangent bundle as a Seifert fiber space. For convenience, we analyze the action while viewing the unit tangent bundle as the representation variety $L(\Gamma)$, as described below.

By definition, the $S^1$ action consists of flowing along the fibers. The action on $UT\H$ is given by rotating all pointers by
a fixed angle while fixing their base points in $\H$. 
Via the homeomorphism $g:B\mapsto(B^{-1}(i),dB_i^{-1}(\pi/2))$ defined in Section \ref{hutb}, The action of $\theta\in S^1$ on $PSL_2(\R)$ is given by $\theta: B\rightarrow K_{\theta}\cdot B$, where $K_{\theta}$ is the rotation about $i$ by $\theta$.
The $S^1$ action in which we are interested on $PSL_2(\R)/\Gamma_0$ is the quotient of the action on $UT\H$. 
We next make use the homeomorphism $h$ described in \ref{representation_variety} between $PSL_2(\R)/\Gamma_0$ and $L(\Gamma)$, $h:B\cdot \Gamma_0\mapsto(B(i), dB_i(\pi/2))$. Now, for $(p,\alpha)\in L(\Gamma)$ choose $B\in PSL_2(\R)$ such that  
$(B(i), dB_i(\pi/2))\sim(p,\alpha)$. The $S^1$ action is thus given by $\theta: (p,\alpha)\mapsto (K_\theta\cdot B(i),dK_\theta\cdot dB_i(\pi/2))\sim(K_\theta (p),\alpha+\theta)$. i.e., the $S^1$ action is given by rotations of all lattices in $L(\Gamma)$ about $i$.
The action takes $\Gamma_0$ back to itself for
$\theta=\frac{2\pi}{k}$, while any lattice with a vertex very close to $i$ will return to itself only for $\theta=2\pi$. Thus, the $S^1$ orbit of $\Gamma_0$ (all lattices with a $k$-cone point at $i$) is the singular
fiber of order $k$. Take a small neighborhood $B_{\epsilon}(i)$ of $i$ in the plane (containing no other
vertices of $\Gamma_0$). The set of lattice representations with a $k-$cone point within $B_{\epsilon}(i)$ is a toral neighborhood
of the $k$-singular fiber. A meridional disk for this torus is the set of such lattices with (say) an upward direction (pointing to the nearest neighbor), and the set of such lattices with a cone point and an upward direction on the circle $S_{\epsilon}(i)$ 
is a meridian. 

Fix a point $(x_0,\frac{\pi}{2})$ on
the meridian, and mark the $k$ points equally spaced along  $S_{\epsilon}(i)$ including $x_0$ by $\{x_0,x_1,...,x_{k-1}\}$, ordered counterclockwise. The lattices corresponding to $\{(x_0,\frac{\pi}{2}),(x_1,\frac{\pi}{2}),...,(x_{k-1},\frac{\pi}{2})\}$ are in the same $S^1$ orbit, and it is the order in which they are transformed to one another which determines $\nu_k$, as in \ref{Seifert}.

The rotation by ${2\pi}/{k}$ takes $(x_0,\frac{\pi}{2})$ to  $(x_1,\frac{\pi}{2}+\frac{2\pi}{k})\sim(x_1,\frac{\pi}{2})$. Hence $\nu_k=1$, and in the same manner $\nu_n=1$.

Concluding, the unit tangent bundle is the union of the Seifert tori $(k,1)$ and $(n,1)$. \end{proof}

There is of course some identification on the boundaries of these two tori. This identification is addressed in the following theorem.

\begin{thm} The unit tangent bundle of $\O_{(n,k)}$ can be embedded in the lens space $L(n+k-nkc, 1-nc)$
for any $c\in\Z$. \end{thm}

\begin{proof} Consider again the two disks into which $\O_{(n,k)}$ is divided in the previous proof, depicted in Figure \ref{division}. These disks intersect along a single (open) segment,  denoted $l$, included in each of their boundaries.
The corresponding tori, which are the unit tangent bundles to each of the disks therefor each include the unit tangent bundle to $l$ on their boundaries. 

By identifying the fibers lying above $l$ with one another, one arrives at a gluing of the two tori $(k,1)$ and $(n,1)$ along an annulus. It is this gluing that yields the geometry of $UT\O_{(n,k)}$.
The annulus is a union of regular fibers. The gluing must identify a regular fiber on one boundary torus to a regular fiber on the other boundary torus. This is true as the fiber above the same point on $l$ appears as a regular fiber in each of the boundaries. 

Consider all  orientation reversing homeomorphisms between the boundary tori. Any such homeomorphism is
given by matrix multiplication, by a matrix in $GL(2, \Z)$ with determinant $-1$. It is easy to compute
that all such matrices satisfying that a $(n,1)$ curve is glued to a $(k,1)$ curve are \bea
M_{n,k,c}=
\left(\begin{array}{cc} 
kc-1 & n+k-nkc \\
 c       & 1-nc \\ 
\end{array}\right) \eea
 for any $c\in\Z$.  Two
solid tori glued along their boundaries so that the meridian of one is glued to a $(p,q)$ curve on the
other results in the lens space $L(p,q)$, as in Section \ref{Lens}. Hence the unit tangent bundle can be embedded into the
$\aleph_0$ lens spaces $L(n+k-nkc, 1-nc)$ as required.
\end{proof}

\begin{cor} $UT\O_{(n,k)}$ can be embedded into $S^3$ if an only if $\{n,k\}=\{2,3\}$, namely, for $\Gamma_{(2,3)}=PSL_2(\Z)$.
\end{cor}

\begin{proof}
First, the lens space $L(p,q)$ is homeomorphic to $S^3$ if and only if $p=\pm1$. Indeed, recall that 
$L(0,1)\cong S^2\times S^1$, $L(p,q)\cong L(-p,q)$ and for $p>1$ $\pi_1(L(p,q))\cong \Z/p$, (see
\cite{Rolfsen} page 234). Thus the unit tangent bundle can be embedded into $S^3$ if and only if there
exists an integer $c$ such that $k+n-knc=\pm1$.  By direct calculation, this can happen only for the case
$(2,3)$ of the modular surface, or for orbifolds with less than two cone points.

As adding more singular fibers can never result in $S^3$, see \cite{Seifert}, this proves the claim.
\end{proof} 

\subsection{Defining the vector field $\mathcal{V}$}

Our next goal is to obtain an homeomorphism from pointers on the orbifold
into any of the relevant lens spaces computed above.  We will now assume that $n=2$, and we do this in two steps. The first consists of defining a particular vector field $\mathcal{V}$ on the orbifold minus a neighborhood of the cone points (in the present section). This gives an embedding of its domain, which is a pair of pants, into the unit tangent bundle. The second consists of completing this embedding to the different possible embeddings of $UT\O_{(n,k)}$ into lens spaces (in the two following sections). The completion is obtained by glueing in first the missing cone point neighborhoods, then glueing in  the torus corresponding to the cusp. The last gluing can be done in $\aleph_0$ different ways, according to the Euler number of the resulting closed three manifold.

We next define a vector field $\mathcal{V}$ on the orbifold with some small neighborhoods
of the cone points removed as in Figure \ref{vector field}:
Take the orbifold by a homeomorphism to the unit 2-sphere punctured at the north pole, taking the cusp to
the north pole and the $k-$cone point to the south pole. Pull back the vector field pointing west at each
point. This can't define a vector field on the orbifold including the neighborhood of
the $2-$cone point, as can be seen from Figure \ref{vector field}. Thus the domain of the vector field is a pair of pants, which we denote by $A$.

\begin{figure}[h]\centering \includegraphics[width=12cm]{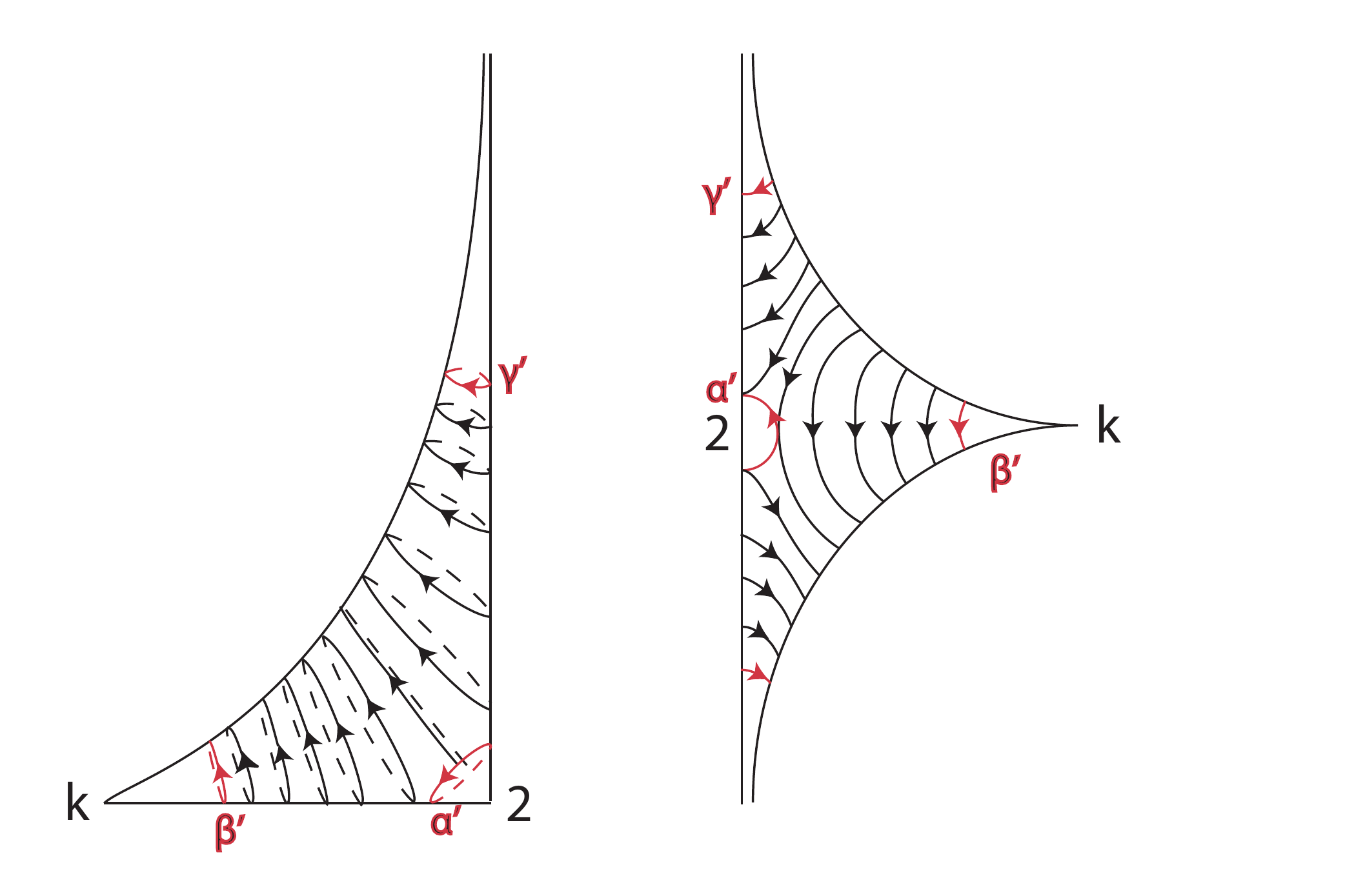} \caption{The vector field $\mathcal{V}$ over
$\O_{(2,k)}$, shown both on the orbifold and on a fundamental domain.}\label{vector field} \end{figure}

Another way of viewing the vector field $\mathcal{V}$ is as the image of an embedding of $A$ into $UT\O_{(2,k)}$, taking each point of $A$ to the pointer based at this point, with the direction given by the vector field at this point. The embedding has the following properties:

\begin{enumerate}

\item Every fiber in the Seifert fibration of $UT\O_{(2,k)}$ intersects $\mathcal{V}$ once, since each fiber consists of the circle of directions at a point, and $\mathcal{V}$ chooses one of these directions. Hence $\mathcal{V}$ is a punctured orbit surface for the Seifert fiber space $UT\O_{(2,k)}$, and in particular,  its boundary curves $\alpha'$, $\beta'$ and $\gamma'$ are crossing curves, i.e. cross each regular fiber on the boundary exactly once.

\item As can be seen in Figure \ref{vector field} , one boundary component $\beta'$ of $\mathcal{V}$ is a small loop around the $k-$cone point
with its tangent vectors, hence is isotopic to the singular fiber $\beta$. This means that on a torus which is the boundary of a tubular neighborhood of
the $k-$singular fiber, $\beta'$ is homologous to a longitude + $x$ meridians for some $x\in\Z$. A
regular fiber on this torus is a $(k,1)$ curve, and by property (1), it intersects $\beta'$ exactly once. Thus, $x$ must be trivial and $\beta'$ is
simply isotopic to a longitude $(1,0)$ on this torus. This uniquely determines the way to glue in the (k,1) singular torus, i.e., determine the way to complete the embedding to the $\beta$ neighborhood.

\item Considering again Figure \ref{vector field}, the boundary component $\alpha'$ of $\mathcal{V}$ is isotopic to a small loop around the $2-$cone point
with a direction rotating relative to its tangent vectors by one full rotation (clockwise). Hence, on the boundary torus of a tubular neighborhood of
the $2-$singular fiber, $\alpha'$ is homologous to a longitude $(1,0)$ minus a regular fiber (which is a $(2,1)$ curve on this torus), plus some number of meridians. Thus $\alpha'\cong(-1,y)$ for some $y\in\Z$. As a
Together with property (1) this yields $\alpha'$ is
isotopic to either minus a longitude $(-1,0)$ or a $(-1,-1)$ curve. The second option turns out to be the correct one, as will be computed below.

\item The third boundary component $\gamma'$ of the surface is isotopic to a small loop around the cusp
together with its tangent vectors.

\end{enumerate}

\begin{figure}[h]\centering \includegraphics[width=5cm]{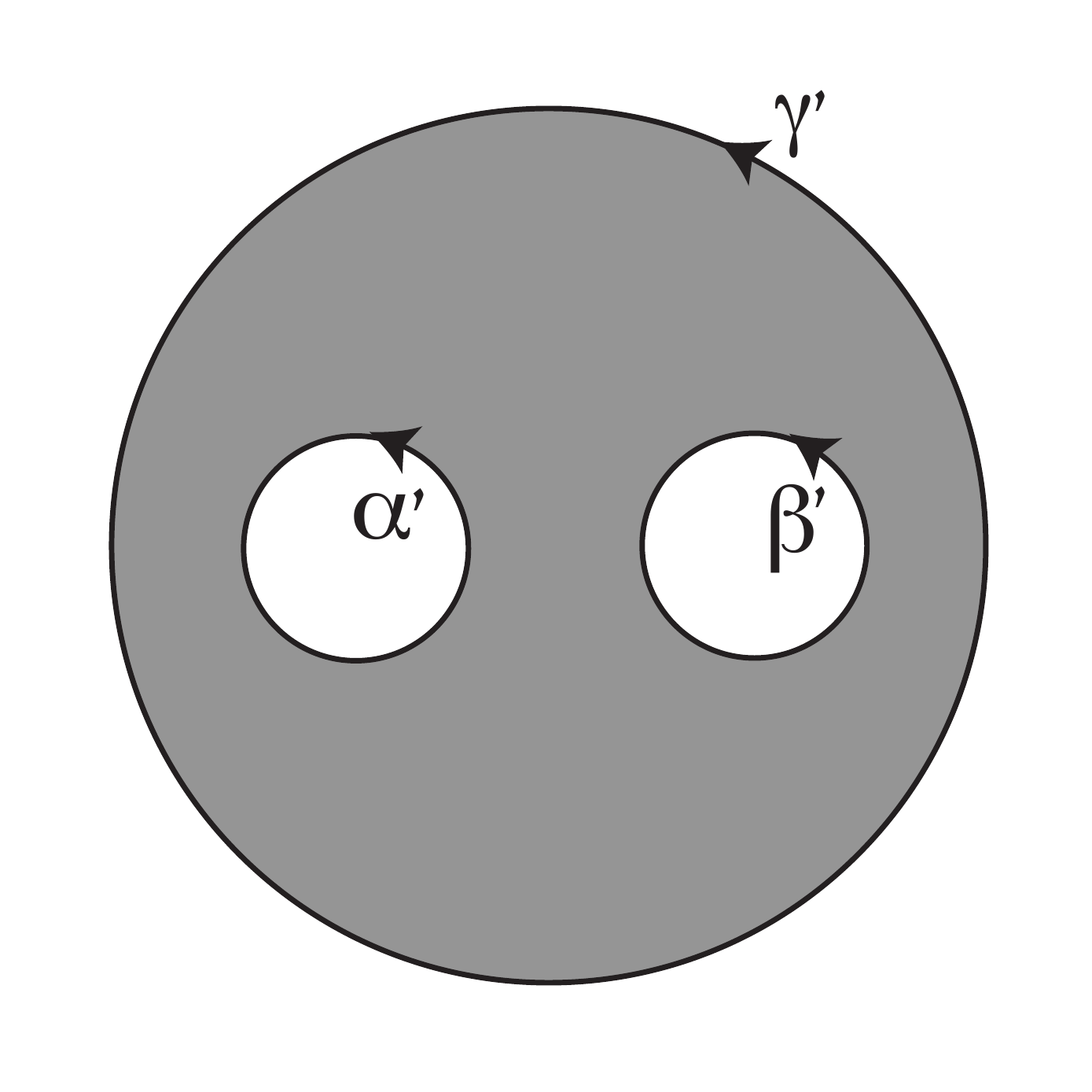} \caption{The vector field $\mathcal{V}$ is a subset of $UT\O_{(2,k)}$ homeomorphic to a pair of
pants with boundaries as above.}\label{pants} \end{figure}

\subsection{The Euler number \label{Euler}}

We now turn to the second step in explaining the embedding of $UT\O_{(2,k)}$ into the Lens spaces. This step consists of explicating the gluing of the missing tori. To begin with, let us note the following conclusion from our previous discussion. 

\begin{lem} All manifolds into which $UT\O_{(n,k)}$ embeds without adding a singular fiber are the Seifert fiber spaces
\bea M=\{O,o,0|b;(k,1),(n,1)\},\ b\in\Z \eea
\end{lem}

\begin{proof} This follows since the pair of pants is orientable, unit tangent bundles are always orientable \cite{UTB_orientable}, $UT\O_{(n,k)}$ has two singular fibers,
and we computed the invariants $(k,1)$ and $(n,1)$ of the two singular tori. 
\end{proof}

Recall any gluing matrix $M_{n,k,c}$ determines an embedding into a specific lens space $M$ and $M$ can be computed directly from the matrix.

We now relate the matrix to the Euler number of $M$. This will determine two things: which of the two possibilities for
the coordinates of $\alpha'$ above is the correct one, enabling us to glue in the $\alpha$ neighborhood; and the relation between the parameter $c\in\Z$ of the
gluing matrix and the Euler number $b\in\Z$ of the lens space.

The Euler number $b$ is determined as follows. An abstract punctured orbit surface, in our case a pair of pants is taken together with its (trivial) unit tangent bundle. A fibered torus with the correct invariants is glued into two holes in the pants, so the boundary of the pair of pants is a crossing curve for these tori. These are the neighborhoods of the singular fibers. Finally, to close the manifold a torus is glued into the third hole in the pair of pants, so that the meridian $\mu$ is glued to the curve $\gamma'-b\cdot f$ where $\gamma'$ is the remaining boundary component of the pair of pants. 

In our case $\gamma'$ is a crossing curve (see Section \ref{Seifert}), hence one can choose the longitude $\lambda$ to be isotopic to a fiber $f$. Thus for $M$ with Euler number $b$, $\gamma'\cong\mu+b\cdot\lambda$. 

Cut the pair of pants determined by the vector field along some curve $\delta$ connecting the $\gamma'$ boundary to the $\beta'$ boundary component.
then on one side of the path $\delta$ pull the surface into a stripe which then goes $-b$ times in the direction of the fibers above $\delta$, and then glued to the other side of $\delta$. This takes $\gamma'\cong\mu+b\cdot f$ to $\tilde\gamma'\cong\mu$. This means that in the manifold $M$ one boundary component of the new pair of pants is a boundary of an embedded disk, by which one can remove that boundary component and arrive at an embedded annulus $A$, with two boundaries: $\alpha'$ and $\beta'-b\cdot f$. Thus in $M$ $\alpha'$ can be isotoped to $-\beta'+b\cdot f$.

The matrix $M_{k,2,c}$ computed before takes the curve 
$\left(\begin{array}{c} -1 \\ -1 \end{array}\right)$ in the $(2,1)$ torus
to the curve
\bea
\left(\begin{array}{c} -1 \\ 0 \end{array}\right)+(c-1)\left(\begin{array}{c} k \\ 1 \end{array}\right)=-\beta'+(c-1)\cdot f\eea in the $(k,1)$ torus.
This decomposition cannot be achieved with the other possible form for $\alpha'$.

This yields,
\bea
\alpha'=\left(\begin{array}{c} -1 \\ -1 \end{array}\right),\eea\bea
b=c-1,
\eea
completing the embedding determined by $\mathcal{V}$.

\begin{rem} The complement of the image of the embeddings of $UT\O_{(n,k)}$ into lens spaces discussed before is a solid fibered torus. This torus is a regular neighborhood of any one of its fibers,i.e., the image of the unit tangent bundle is the complement of a regular fiber in the lens space. Denote this fiber by $\xi$. The
missing fiber, or `missing knot' $\xi$, plays the same role as the trefoil for the modular flow. 
\end{rem}


\subsection{Interpreting the embedding}

We next make the embedding much more explicit: for any given pointer in $UT\O_{(n,k)}$ we will be able to identify its image in the lens space.

Recall each of the relevant lens spaces are the union of a $(k,1)$ and a $(2,1)$ tori.
Take the $(k,1)$ torus to be a very small neighborhood of the
singular fiber $\beta$ so that $\beta'$ is on its boundary. It then follows that the $(2,1)$ torus contains the rest of $UT\O_{(n,k)}$, and in particular the vector field. 
We denote the $(2,1)$ torus by $\T_2$ and take this torus to be oriented in the usual way, so that the orientation of the meridian followed by the orientation of the longitude gives the orientation of $\partial(\T_2)$. The missing fiber $\xi$ corresponding to the cusp is then a regular fiber in the interior of this torus (and not on its boundary as before).

We determined the coordinates of the boundary curves $\alpha'$, $\beta'$ and $\gamma'$ of  the vector field. The coordinates are given each on the boundary tori which are the neighborhoods of  $\alpha$, $\beta$ or $\xi$. Any two fixed boundary curves determine the pair of pants up to isotopy, see Seifert \cite{Seifert} page 44. Hence, it suffices that we identify one such pair of pants in the (2,1) torus, in order to identify $\mathcal{V}$. 

We now describe such a pair of pants for the case $b=0$. For this case $\gamma'\cong\mu$ is a meridian and so when filling in $\mu$ with a disk one gets an annulus intersecting $\xi$ transversally. This corresponds to gluing the tori comprising the lens space by the matrix $M_{k,2,1}$ (since $c=1-b$). We represent the $(2,1)$ solid torus without a neighborhood of its core $\alpha$ as a union of concentric tori. Consider a $(-1,-1)$ curve on any one of these concentric tori. It intersects any $(2,1)$ fiber once.  As we saw in \S \ref{Euler}, $\alpha'$ is a $(-1,-1)$ curve, and hence there is an annulus connecting a $(-1,-1)$ curve $\delta$ on the inner boundary of the torus $T_2$ to $\alpha'$,  through $(-1,-1)$ curves on each of the concentric  tori. The curve $\xi$ appears in $T_2$ as a $(2,1)$-curve on one of the concentric tori and therefore will intersect this annulus transversally at a point. Thus, by puncturing the annulus at that point we arrive at the desired pair of pants. This follows from uniqueness since the punctured annulus  has the desired boundaries $\alpha'$ and $\gamma'$. It follows that the pair of pants also has $\beta'$ as a boundary, and indeed, one can easily check that the $M_{k,2,1}$ matrix takes $-\delta$ to $\beta'$:

\bea
\left(\begin{array}{cc} 
k-1 & 2-k \\
 1  & -1 \\ 
\end{array}\right) \left(\begin{array}{c}
1\\
1\\
\end{array}\right)=
\left(\begin{array}{c}
1\\
0\\
\end{array}\right)=\beta'.
\eea

\begin{figure}[h]\centering 
\includegraphics[width=5cm]{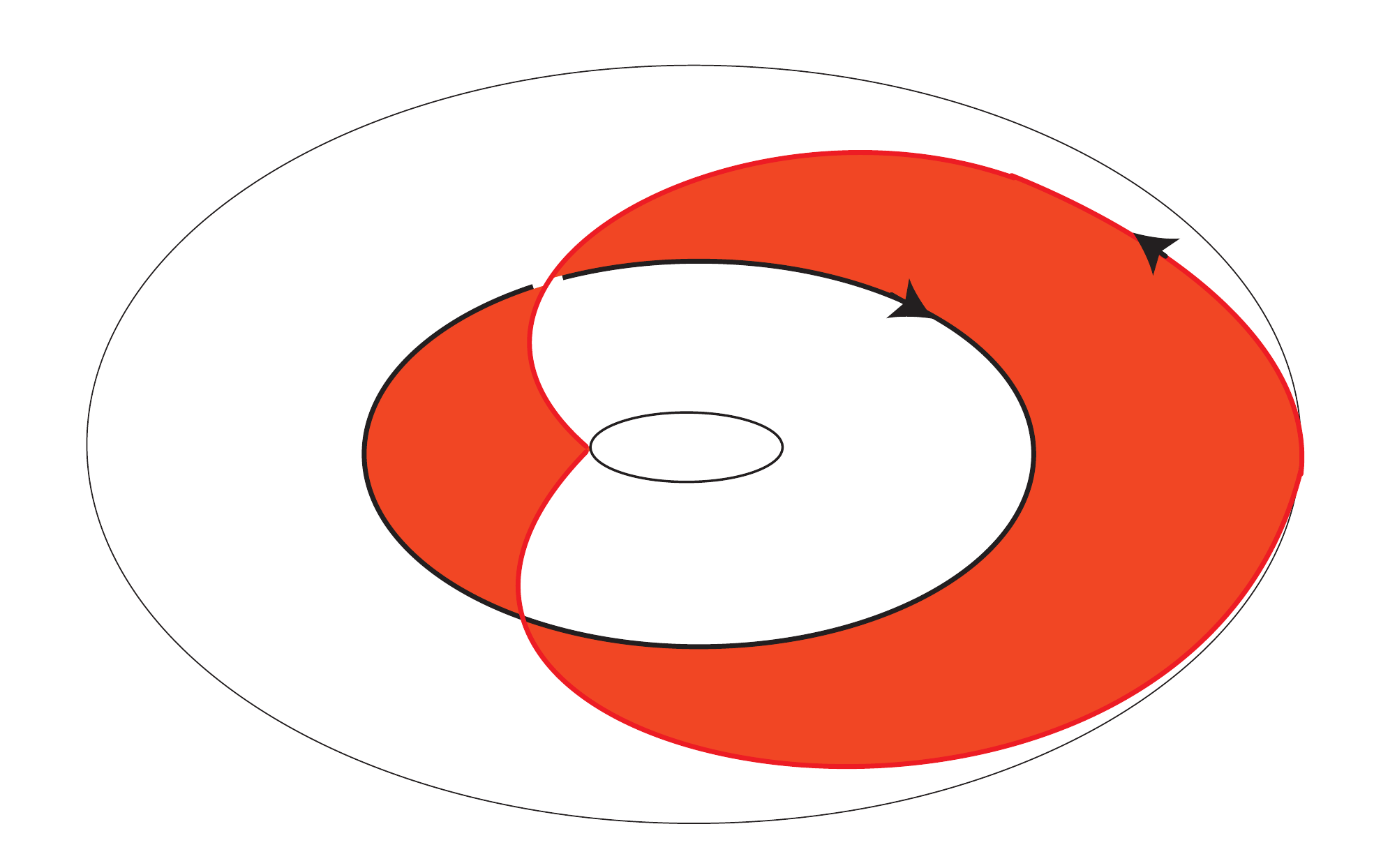}
\caption{Two of the boundaries of the pair of pants corresponding to the vector field, the third being the puncture by the missing fiber $\xi$.}\label{embedded_pair}
\end{figure}

We continue our analysis of the embedding  (which is determined up to isotopy) from $UT\O_{(2,k)}$ into the lens space, focusing 
on embedding the complement of small neighborhoods of the singular fibers into $\T_2$.

To begin with, our discussion in the previous section shows that the pair of pants described above (and depicted in Figure \ref{embedded_pair}) is the image of all  pointers comprising the vector field. Each pointer  can be rotated by any angle while fixing its base point. By rotating all pointers of the vector field by the same angle at the same time we get an isotopy of the entire pair of pants $\mathcal{V}$. Of course, when the angle is $2\pi$ we retrieve $\mathcal{V}$ once again,  as in Figure \ref{rotation_pair}. Naturally, the intermediate punctured orbit surfaces all have the same Seifert invariants and thus are determined by $(-1,-1)$ curves.

\begin{figure}[h]\centering 
\includegraphics[width=5cm]{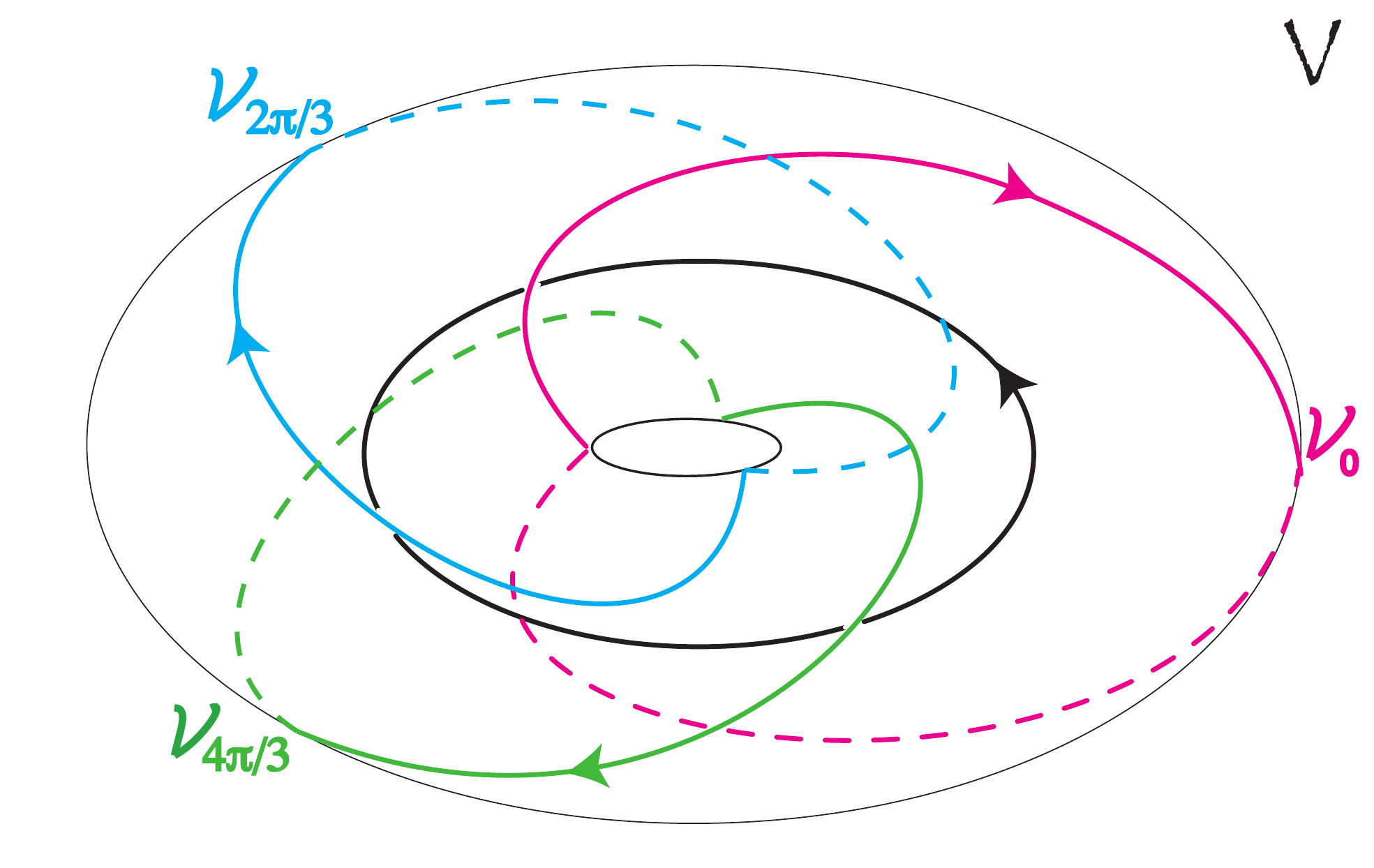}
\caption{Boundaries of three isotopic pairs of pants, the boundaries near the alpha core are perpendicular to the outer ones, and are sketched as one curve. The third boundary component is the puncture by $\xi$ . The curve marked $\mathcal{V}_0$ is the boundary of the original vector field $\mathcal{V}$. Every point travels along its fiber as one rotates the vector field counter clockwise, yielding the other pairs of pants.}\label{rotation_pair}
\end{figure}

The set of these pairs of pants is determined up to isotopy. Recall also that a counterclockwise rotation corresponds to flowing along the fibers in the positive direction.
This yields a concrete mapping $(p,\theta)\mapsto\T_2$ as desired, where $\theta$ is the angle relative to the vector field.


\section{Templates \label{templates}}

\subsection{The Birman Williams theorem (\cite{Birwil2}, \cite{book})\label{Birman_Williams}}

\begin{Def} A \emph{template} is a compact branched two-manifold with boundary and a smooth expanding semiflow built
from a finite number of branch line charts, as in Figure \ref{branch_chart} \end{Def} 

\begin{figure}[h]\centering \includegraphics[width=5cm]{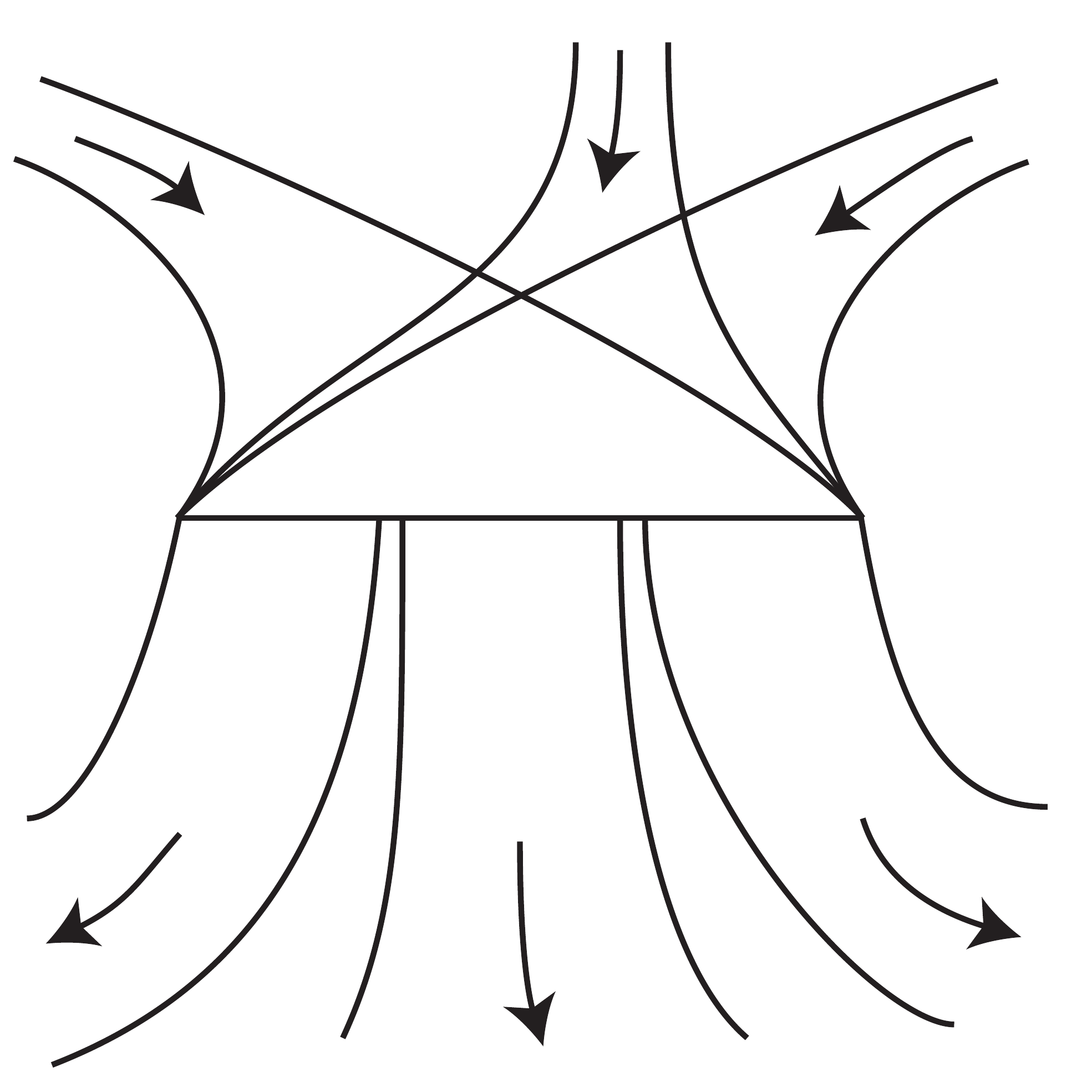} \caption{}\label{branch_chart}
\end{figure}

We now formulate the following fundamental fact, which underlies our discussion. 
\begin{thm}[Birman \& Williams]\label{BWthm} Given a flow $\phi_t$ on a three-manifold
$M$ having a hyperbolic chain-recurrent set, the link of periodic orbits $\mathcal{L}_{\phi}$ is in bijective correspondence with the link of
periodic orbits $\mathcal{L}_{\mathcal{T}}$ on a particular embedded template $\mathcal{T}\subset M$. On any finite sublink, this correspondence is via ambient isotopy.
\end{thm}

In our case of geodesic flows on hyperbolic orbifolds, the chain recurrent set equals the recurrent set, and this set is the closure of the set of closed geodesics. The recurrent set in this case as the entire system is always hyperbolic due to
the divergence of geodesics in $\H$: Hyperbolicity means that transversal to the direction of the flow there is an expanding direction, and a contracting direction. These two directions are given by the horocycles, as explained in \S \ref{Geodesic_flows} for the geodesic flow on the hyperbolic plane. Since hyperbolicity is a local property, it descends to the quotients of the system, the flows on hyperbolic surfaces and orbifolds. Therefore by Theorem \ref{BWthm} there exists a template for the flow on any 
hyperbolic orbifold.

The general method of obtaining a template consists of collapsing the stable manifolds of the system. Thus one gets a semi-flow on two dimensional system, while all periodic orbits consist as different periodic orbits never belong to the same stable manifold.
The problem is that in general the stable manifolds will be dense and collapsing them would result in a non-Hausdorff space. Very informally, Birman and Williams overcome this by first "separating" the stable manifolds by preforming a surgery on one or two periodic orbits, reducing the dimension of the recuurent set to one, and then finding a nice neighborhood to this set in which the stable manifolds can indeed be neatly collapsed. This process is not constructive for a general flow, and thus obtaining templates for different flow is an interesting problem, and templates have been obtained for a limited number of flows. 

For hyperbolic geodesic flows the stable and unstable manifolds are known as in \S \ref{Geodesic_flows}. We deal with the case of an orbifold with a cusp, for which the dimension of the basic set can always be reduced: 
by opening the cusp as in \S \ref{representation_variety} one arrives at system in which the set of periodic orbits is already not dense, and so the stable direction can simply be collapsed. Ghys uses this fact to obtain the template for the modular surface, a further discussion of this matter can be found in \cite{Pierre}. Thus for our case there is a better understanding of obtaining the template.
The details are described in the remaining of this section where we construct the templates
for our examples.

\subsection{Constructing a template for $\O_{(2,k)}$}

We now describe the embedded template in $\T_2$. 
Following Ghys \cite{Ghys}, consider the $k$ fold cover of the fundamental domain of $\O^d_{(2,k)}$ as in Figure \ref{domain}, denoted $D$, where the generators of the group act by a rotation by $\pi$ about $x$, and a rotation by $2\pi/k$ about $y$  (see \S \ref{representation_variety}.

\begin{figure}[h]\centering
\includegraphics[width=5.5cm]{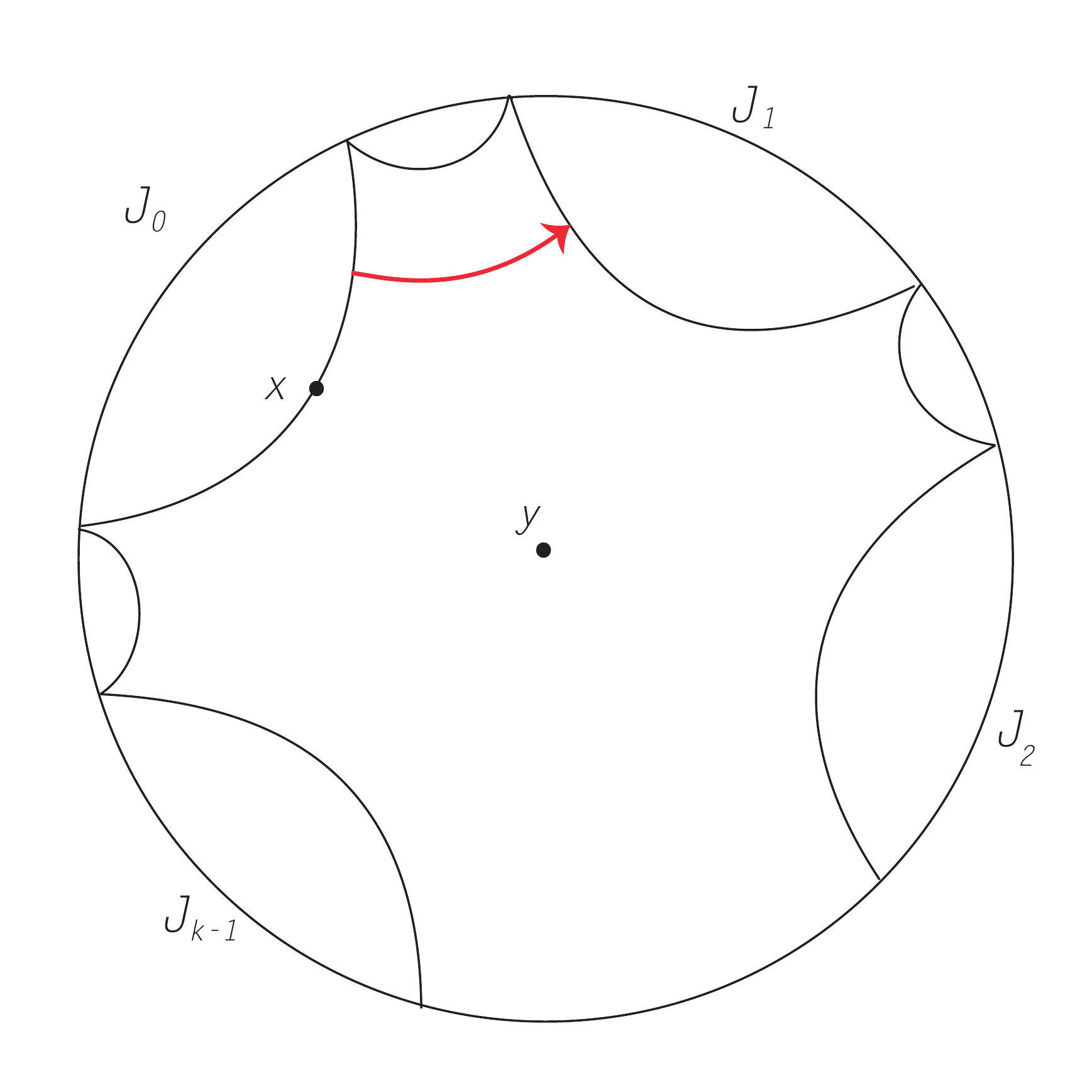}
\caption{}\label{domain}
\end{figure}

For every closed geodesic there is a lift passing through $D$, crossing from some segment of $\{J_0,\hdots,J_{k-1}\}$ to another. By using the rotational symmetry about $y$, any closed geodesic has a lift with an arc emanating from $J_0$ and crossing the domain. At the endpoints of the arc, where the lift leaves $D$, we use again the rotation about $y$ and then the rotation about $x$ to identify each endpoint with a starting point of another arc. Hence these arcs contain the recurrent set of the geodesic flow.

We can choose the segment along the boundary of the fundamental domain  (depicted in Figure \ref{branch_line}) with perpendicular direction vectors into the domain, as a single branch line for the template. 

\begin{figure}[h]\centering \includegraphics[width=5cm]{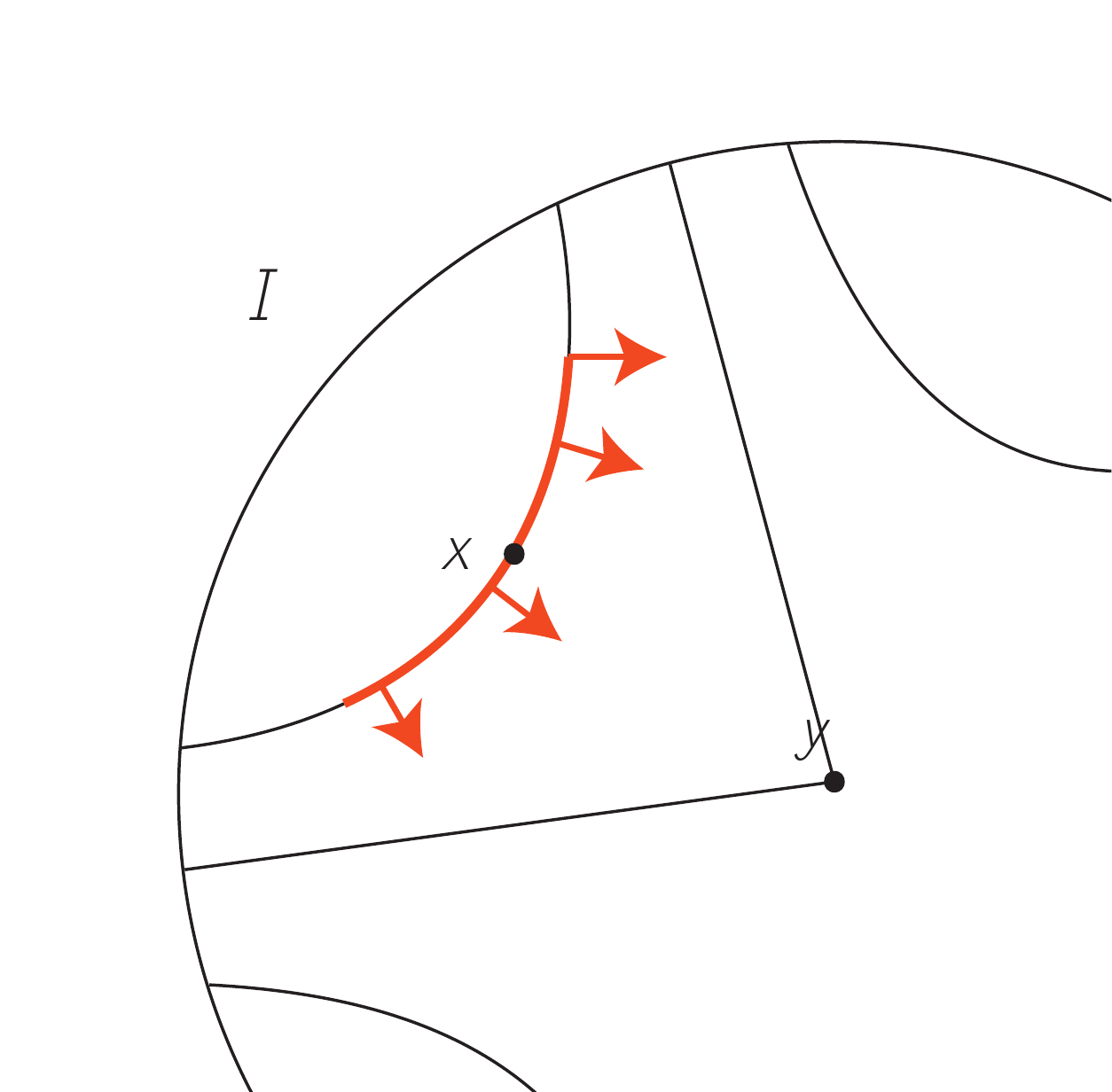} \caption{}\label{branch_line}
\end{figure}

For a general geodesic that is not necessarily perpendicular when entering the domain, we can always choose a different point along the stable manifold of the entry point of the geodesic, and arrive at a geodesic that is perpendicular at its entry point. Recall the stable manifold is the horocycle corresponding to the starting point at infinity (on $J_0$) of the geodesic, with the direction vectors perpendicular to the horocycle- pointing away from the starting point. If the geodesic is already perpendicular at the entry point, the horocycle and the boundary of the domain are tangent, and we do nothing. Else, the horocycle crosses the boundary transversally and enters the domain. The horocycle returns to $J_0$, and therefore must cross the boundary again at some other point. By looking at the geodesics perpendicular to the horocycle at each point between the two crossing points, we see the angle they create with the boundary changes monotonically, and must pass through $\pi/2$. Hence, by the mean value theorem there exists a point as required. Thus, the chosen branch line contains a point of the equivalence class of every arc passing in $D$, and so catches all the recurrent dynamics.

Using the embedding derived in the previous section, we are able to embed the branch line in $\T_2$. One half of the branch line is contained in $\mathcal{V}$, while the other half in $\mathcal{V}_{\pi}$. Hence, the branch line is embedded as shown in Figure \ref{embedded_cores} for the particular case $(2,5)$. 

Our template has $k-1$ stripes emanating from the branch line, each containing the arcs reaching the same segment $J_i$ at the endpoint.  The stripes each return to the branch line by using the group symmetries as above. These stripes are called ears, and we denote the ear reachng $J_i$ by $E_i$.  Each ear stretches across the entire branch line when returning to it, as any starting point can be obtained by the symmetries, from endpoints at any segment $J_1,\hdots,J_{k-1}$.  Denote the core of $E_i$ by $c_i$. $c_i$ is the unique closed geodesic contained in the ear. 

To understand the embedding of the template we have to understand the embedding of each core, and the twisting of the ear about its core. We begin with embedding the cores. To this end consider Figure \ref{cores} comparing the tangent vectors to each of the ear cores to the vector field $\mathcal{V}$. Here we draw $\mathcal {V}$ by taking the $k$-fold cover of $\mathcal{V}$ as depicted in Figure \ref{vector field}.
\begin{figure}[h]\centering \includegraphics[width=7cm]{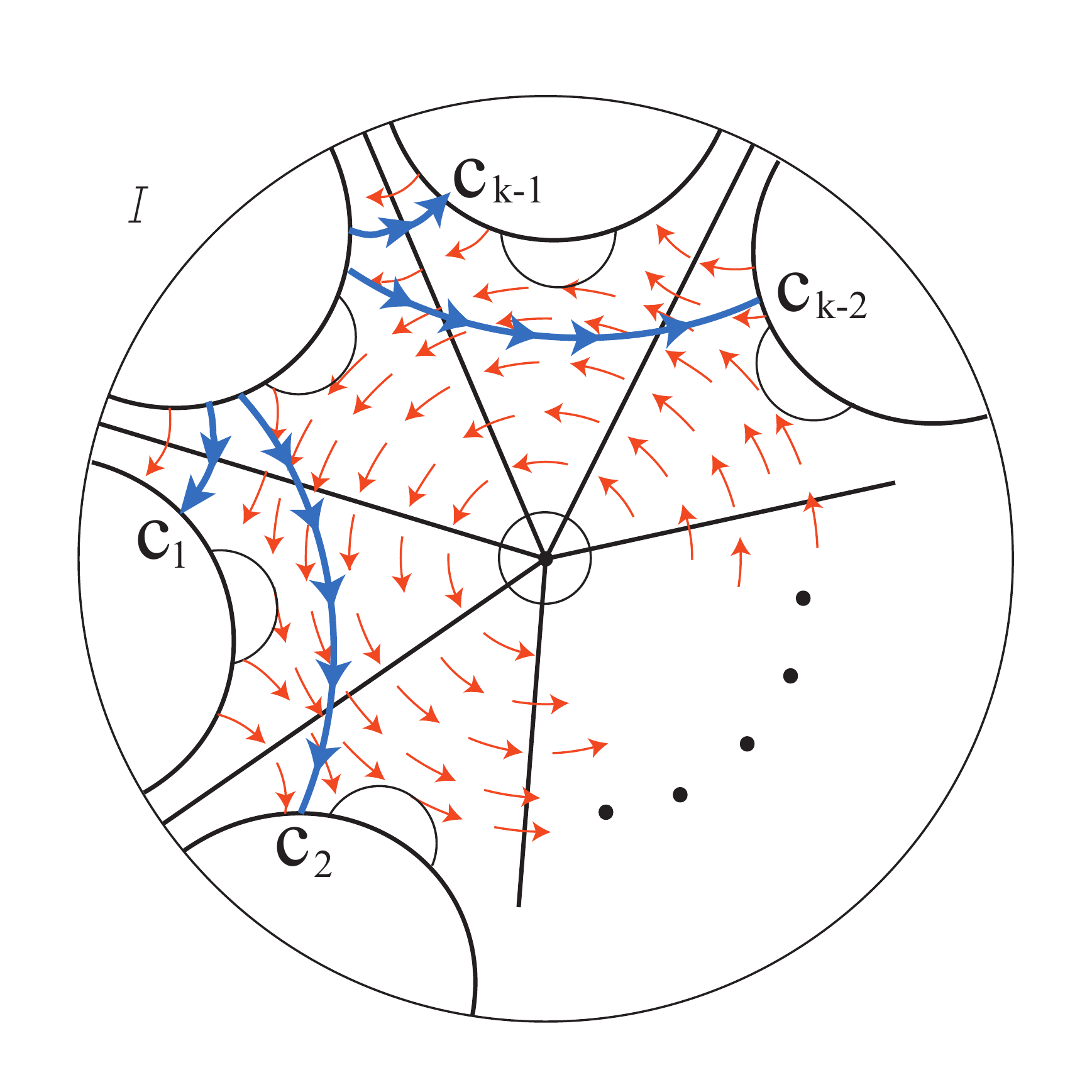} \caption{}\label{cores}
\end{figure}

Now, by considering at the same time this figure and Figure \ref{rotation_pair}, this yields the embedding of the cores into $\T_2$. Together with the branch line, this is sketched for the case $(2,5)$ in Figure \ref{embedded_cores}. For obtaining this figure, the cores $c_1$ and $c_4$ (and in general $c_{p-1}$) are easy do identify as small loops around the missing fiber $\xi$, as on the orbifold they are small loops around the cusp, and they are contained in $\mathcal{V}$ and $\mathcal{V}_\pi$ respectively by Figure \ref{cores}. 
The next two cores can be understood as they follow loops further from the cusp, then around the 2-cone point, then again around the cusp, while traveling back and forth through the different pairs of pants, as their angle relative to $\mathcal{V}$ increases, then decreases, and vice versa, as in Figure \ref{cores}. This is enough for understanding the four cores of the $(2,5)$ case, given in Figure \ref{embedded_cores}, or for these four loops in ane $(2,k)$ template, $k>3$.

\begin{figure}[h]\centering \includegraphics[width=10cm]{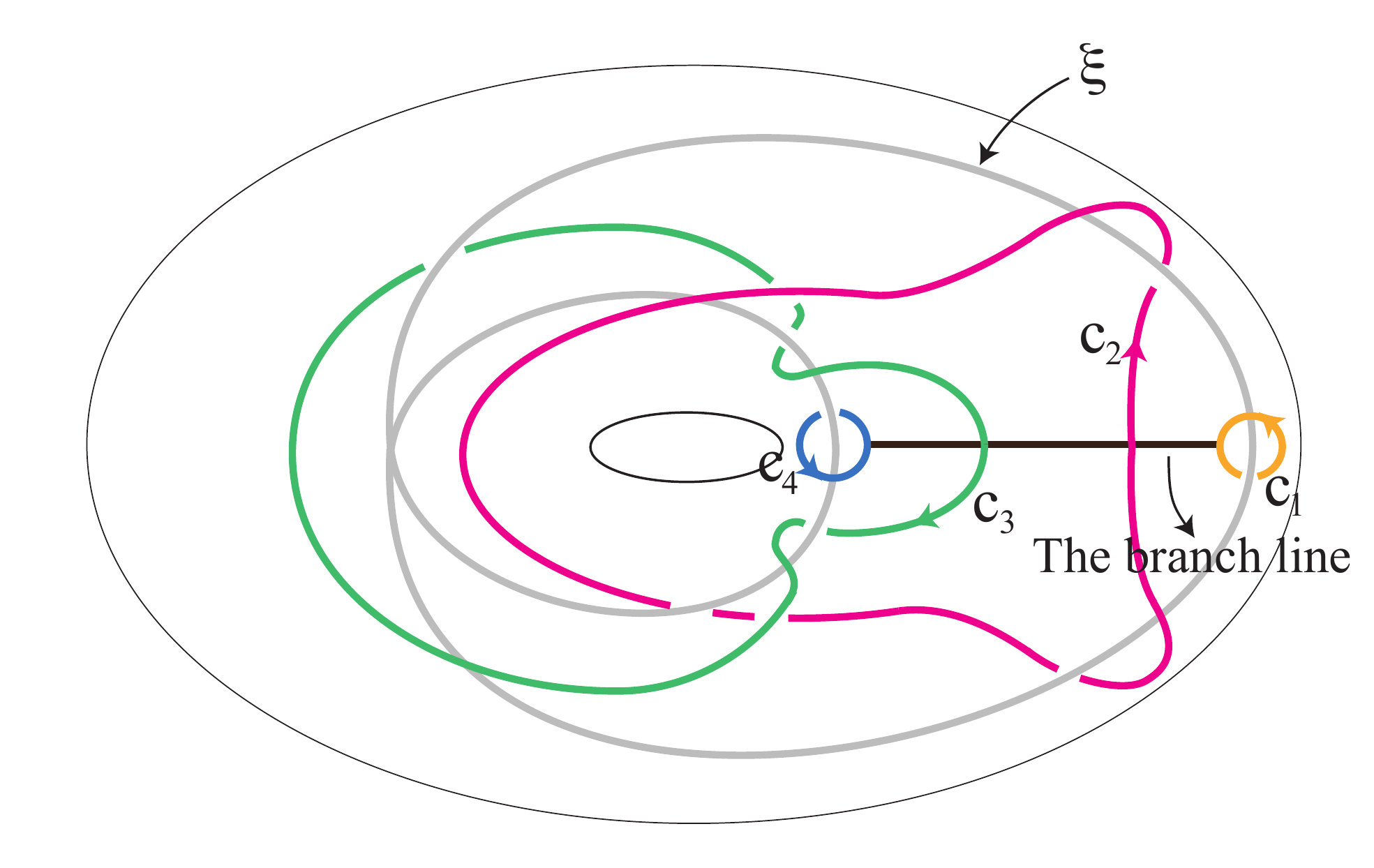} \caption{}\label{embedded_cores}
\end{figure}

We now consider the next ears in a general $(2,k)$ template. The four first ear cores for any flow on $\O_{(2,k)}$ for $k\geq9$ is given in Figure \ref{next_cores}. We note two facts regarding these cores. Each subsequent core  reaches a larger angle relative to the vector field before returning back to the vector field direction.  The core corresponding to the ear $E_j$ has $j$ parts each isotopic to $\beta'$, that is a circle around the $k$-cone points, and each subsequent core is closer to the $\beta'$ curve. These four first loops for any $k$ are sketched, by the same method as for the to first (and last) ears above, in Figure \ref{four_loops}. 

\begin{figure}[h]\centering 
\includegraphics[width=2.5cm]{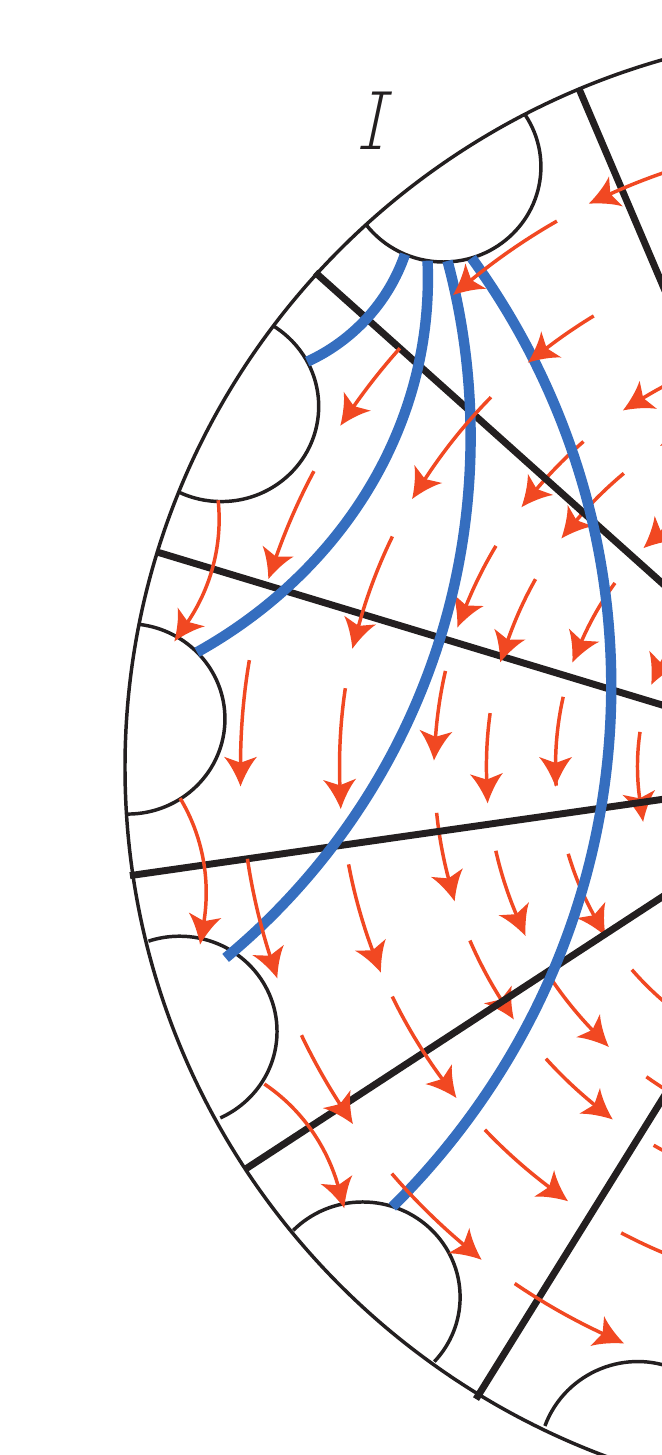}  
\caption{}\label{next_cores}
\end{figure}

\begin{figure}[h]\centering 
\includegraphics[width=10cm]{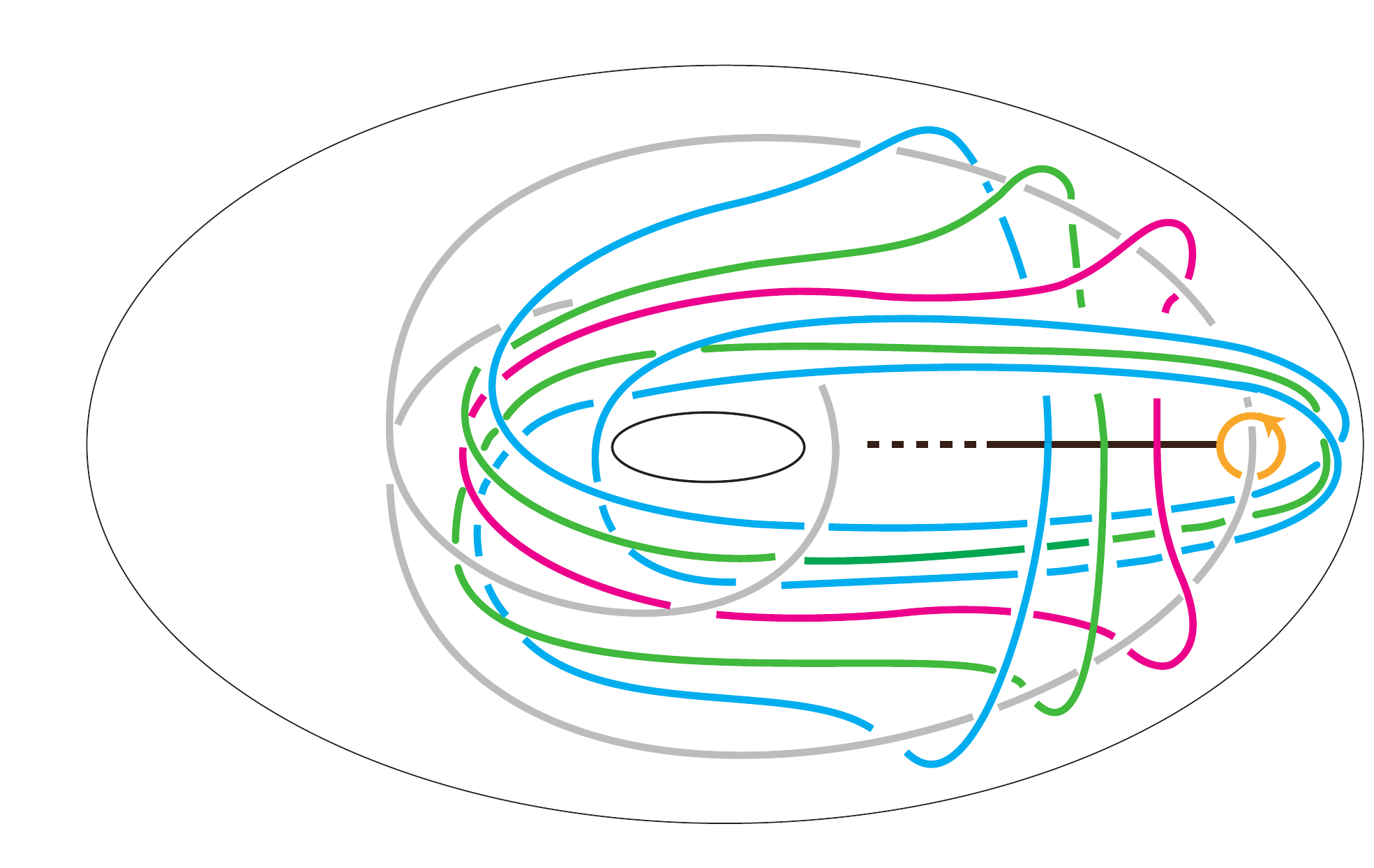}  
\caption{}\label{four_loops}
\end{figure}

We next compute for each ear the embedding of one orbit in addition to the core. This will determine the twists of each ear, and thus the template. 
A second orbit in the same ear is never closed, and so  will have two parts. The first part is a geodesic segment: We choose a geodesic which emanetes from the branchline perpendicularly to the boundary of the domain, very close to the closed geodesic, closer to the 2-cone point along the branch line. For this segment for any of the ears, we find it reaches a larger angle than the core relative to the vector field before returning to it, and is closer than the core to $\beta'$. This is shown in Figure \ref{first_loop} for the first ear. This means, this segment is in the direction of the next core (if it exists), and we can draw this segment it $T_2$ according to our embedding of the cores above.

When the above geodesic segment reaches again the boundary of the fundamental domain as on Figure \ref{branch_line}, its direction is not perpendicular to the boundary. This means that in the unit tangent bundle this geodesic arc did not return to the branch line.
The second part of the orbit will connect it to the branch line through segments of stable manifolds and geodesic segments, as explained above in general. In this case, by choosing a segment close enough to the closed geodesic, the second step can be done in one "move". This is sketched in Figure \ref{first_loop} for the first ear. 

This is the last ingredient needed for determining the template completely, and is then done for each of the template ears.  The resulting template is sketched in figure \ref{embedded_template} for the $(2,5)$ case.

\begin{figure}[h]\centering \includegraphics[width=5cm]{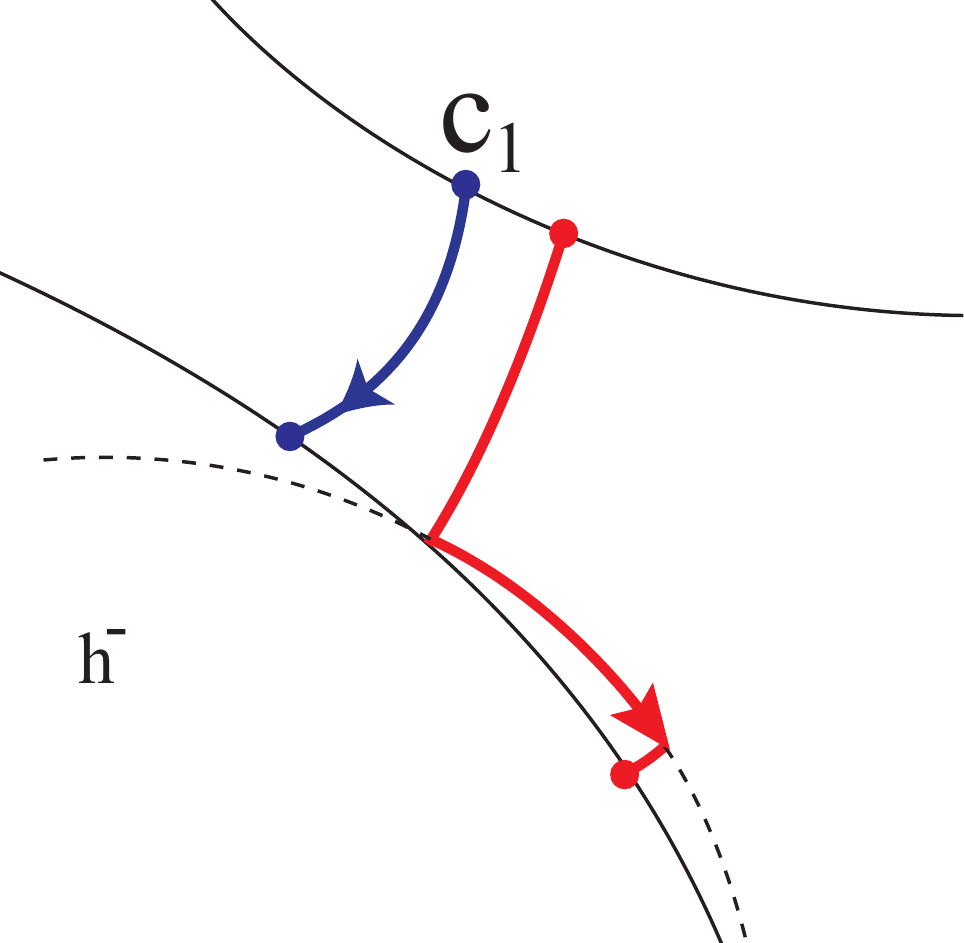} \caption{}\label{first_loop}
\end{figure}

\begin{figure}[h]\centering \includegraphics[width=9cm]{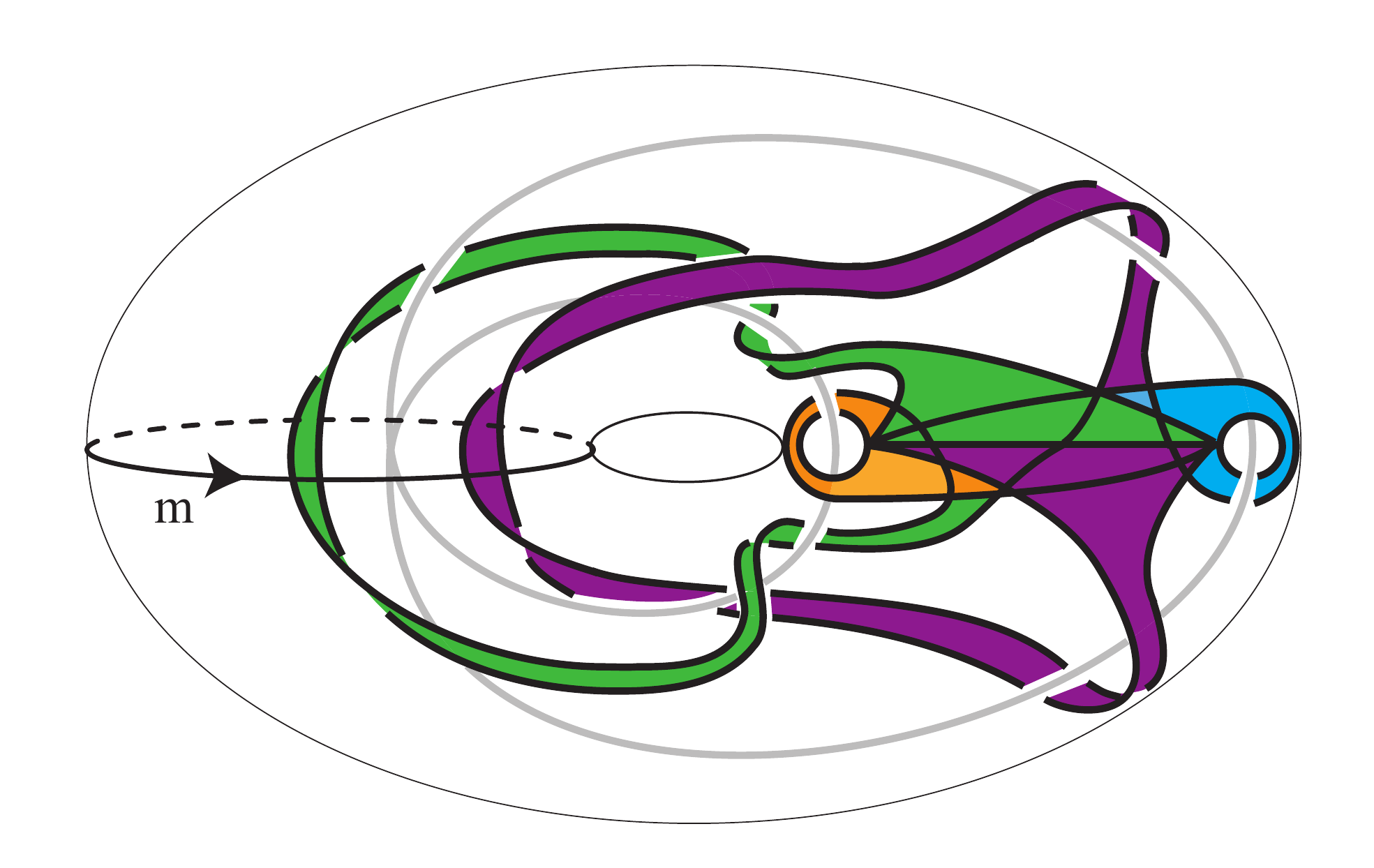} \caption{}\label{embedded_template}
\end{figure}

\begin{rem}
Note that the linking number of any closed geodesic with the missing knot $\xi$ is well defined in any of the Lens spaces, and can be easily read off the sequence of ears of the template the geodesic passes through. 
\end{rem}

We next analyze the structure of the template for a larger $k$. The four loops in Figure \ref{four_loops} can be isotoped to yield Figure \ref{four_loops_isotoped}. It is obvious that this is similar for the four last ears by the symmetry, and that this generalizes to the next ears.  It is true for all ears that for a second orbit starting closer to the middle of the branch line, the maximal angle from the vector field increases, while the image approaches $\beta'$. The number of times the geodesic encircles the $k$-cone point increases by one when passing from one ear to the next. Thus, when one executes the same analysis for the twisting of the ears as in the $(2,5)$ case by examining one more loop in each ear, one finds the additional ears can also be put one beside the other exactly as in Figure \ref{four_loops_isotoped}.

\begin{figure}[h]\centering 
\includegraphics[width=10cm]{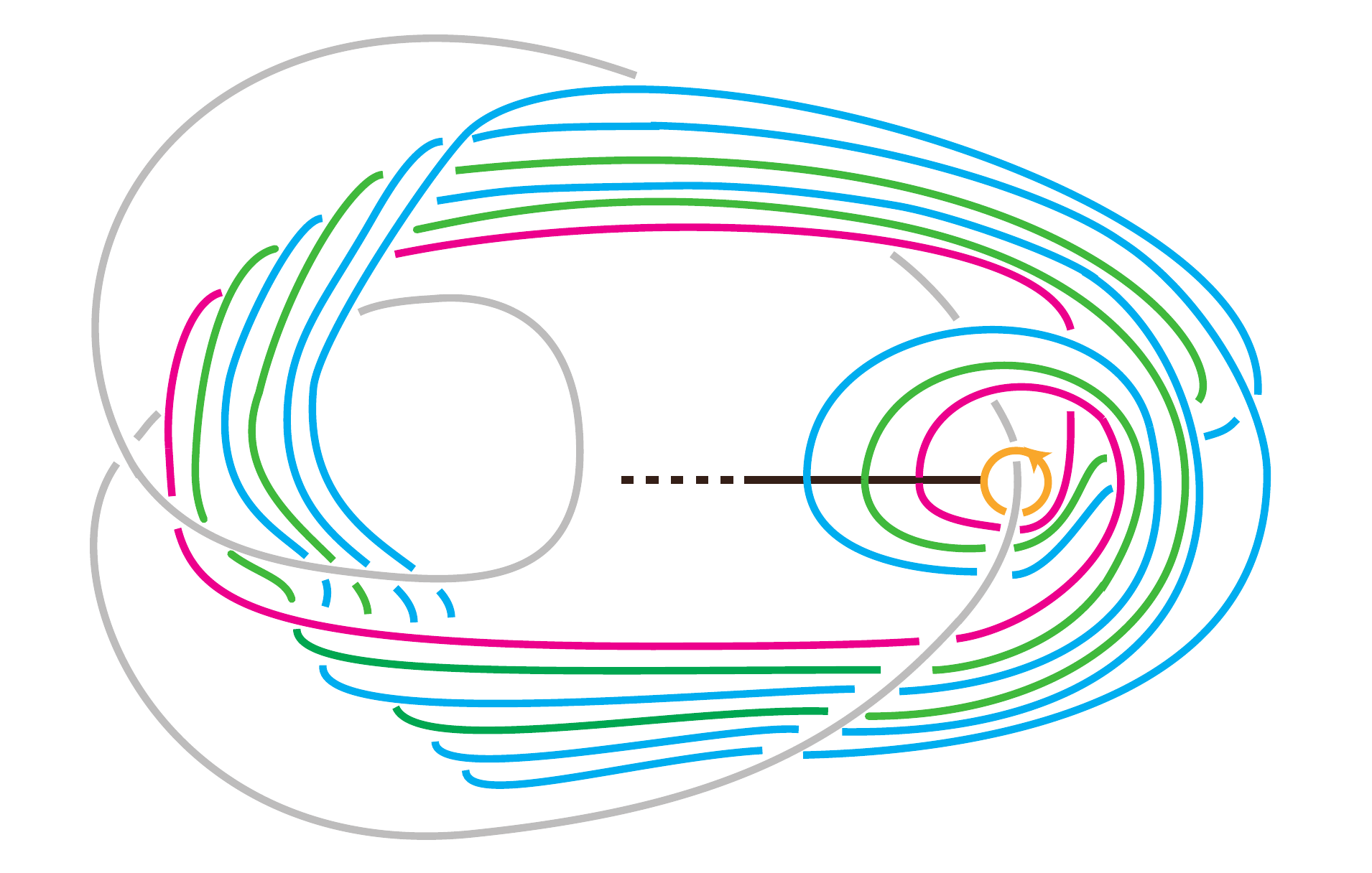}  
\caption{}\label{four_loops_isotoped}
\end{figure}

The last crucial step in our analysis is to glue the ears together along their boundaries, where they run parallel to each other. In this way we obtain a template with a smaller number of ears, that actually does not increase with $k$. As can be seen from Figure \ref {four_loops_isotoped}, we have to add another branch line in the middle of the first loop, and the same is of course true for the last loop, as in Figure \ref{unified_template}.

For each of the three branch lines we now have two ears arriving to it, but it is true only for the middle one that both ears cover it completely. We now address this last fact. Any ear $E_i$ in any of the original templates we constructed above, starts at the central branch line, arrives at one of the side branch lines and then makes some travels along the core of $T_2$ (through one of the long ears in Figure \ref{unified_template}. The number of the travels through the long ear equals $i-1$ if $i<k-1/2$, and $k-i-1$ otherwise. The $i$-{th} ear then passes to the shorter ear emanating from the side branch line to arrive back at the central branch line. Hence, the longer ears in the $(2,k)$ template covers itself only partly when it comes back to the branch line, so that the maximal number of times an orbit can pass through it without arriving back at the central branch line is $k-1/2$. Thus, we find all $(2,k)$ templates, with $k>3$ are in fact very similar, the only change being the length of the two smaller branch lines which is covered by the longer ear.
If one insists on having every incoming ear covering the entire branch line in the definition of a template, one is led to propagate forwards the end point of the long ear, from the branch line forward through the same ear. This has to be done recursively a number of times, and will result in a template very similar to the original template we obtain for the $(2,k)$ case, with a number of ears increasing with $k$.
 In any case we have proved

\begin{thm}
The template for the geodesic flow on the orbifold $\mathcal{O}_{(2,k)}$ for odd $k$, is  a subtemplate of the following template, embedded in the $(2,1)$ torus in the lens space $L(k-2,1)$ .

\begin{figure}[h]\centering \includegraphics[width=9cm]{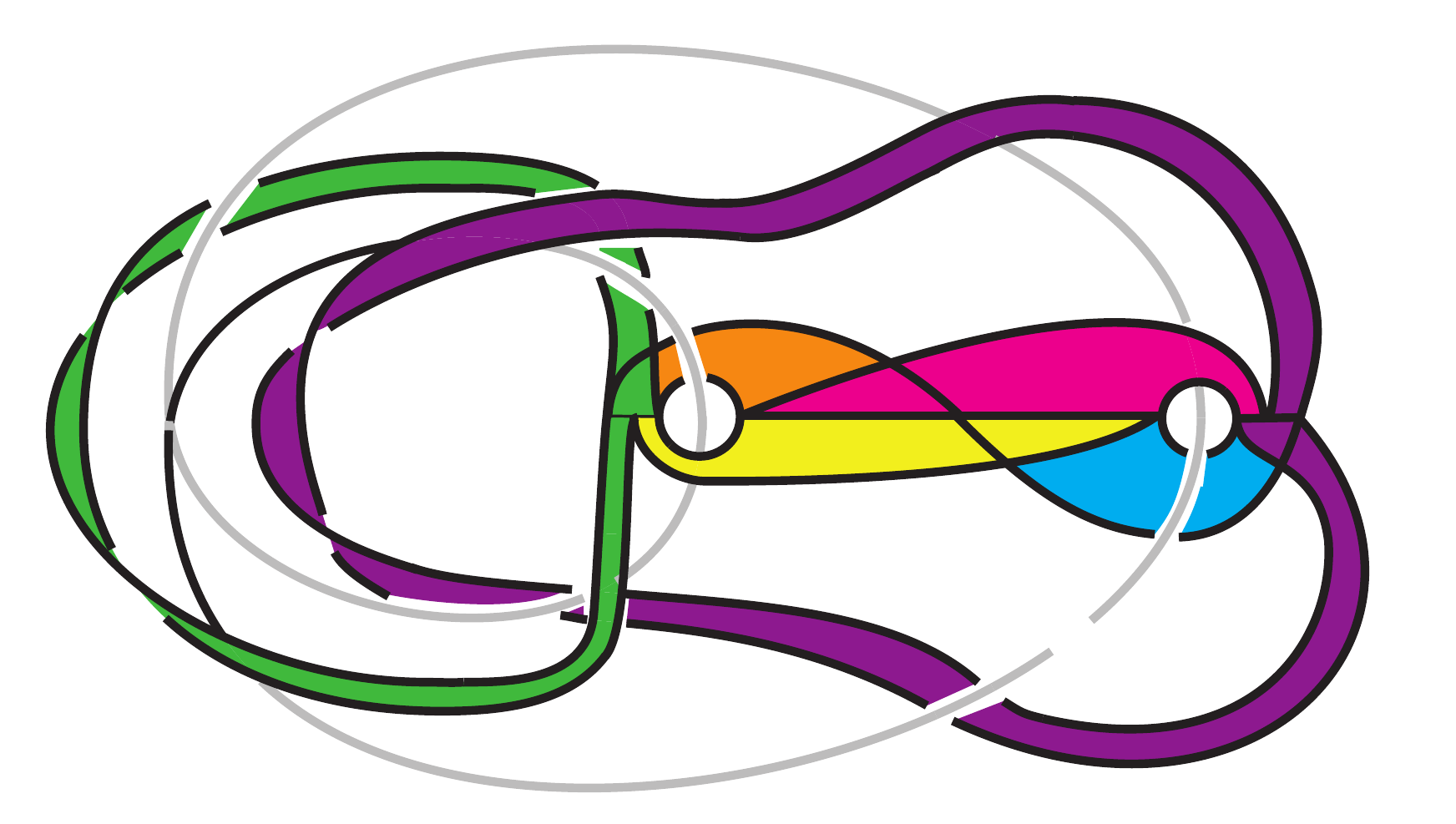}
\label{unified_template}\caption{}
\end{figure}
\end{thm}



\end{document}